\documentclass[12pt]{article}
\usepackage{amssymb,latexsym, amsmath, amsthm, amsfonts, amssymb, enumerate}
\usepackage[notref,notcite]{showkeys}
\usepackage[usenames,dvipsnames,svgnames,table]{xcolor}
\usepackage{color}
\usepackage{graphicx, overpic}

\def\T{{\mathrm{T}}}
\def\X{{\mathrm{X}}}
\def\x{\widetilde{x}}
\def\xNx{\mathrm{x}}
\def\xx{\mathcal{X}}
\def\a{\widetilde{a}}
\def\s{\mathbb S}
\def\Y{{\mathrm{Y}}}
\def\G{\mathbb G}
\def\R{\mathbb R}
\def\N{\mathbb N}
\def\C{\mathbb C}
\def\Z{\mathbb Z}
\def\D{\mathcal D}
\def\U{\mathbb U}
\def\u{\mathrm U}
\def\H{{\mathbb H}(2n , m)}
\def\M{\mathbb M}
\def\OM{\widetilde{\M}}
\def\He{\mathrm{Hess}}
\def\OA{\Omega_*}
\def\OB{\Omega_+}
\def\OC{\Omega_{-, 4}}
\def\RN{{\mathbb R}^2_{>}}
\def\RS{{\mathbb R}^2_{<, +}}
\def\MD{\Lambda}
\def\K{\mathrm{K}}
\def\V{\vartheta_1}
\def\F{\mathcal{R}}
\def\FS{\mathfrak{R}}
\def\JA{\mathrm{J}_{\Lambda}}

\newtheorem{theo}{Theorem}[section]
\newtheorem{lem}{Lemma}[section]
\newtheorem{remark}{Remark}[section]
\newtheorem{cor}{Corollary}[section]
\newtheorem{prop}{Proposition}[section]

\title{The Carnot-Carath\'eodory distance on $2$-step groups}
\author{Hong-Quan Li}
\date{}
\begin{document}

\renewcommand{\theequation}{\thesection.\arabic{equation}}
\setcounter{equation}{0} \maketitle

\vspace{-1.0cm}

\bigskip

{\bf Abstract.} Combining Varadhan's formula, Loewner's theorem with the method of stationary phase, we study the exact formula of the Carnot-Carath\'eodory distance on $2$-step groups. The method is also adapted to determine all normal geodesics from the identity element to any given point (up to a set of measure zero). As an application, we characterize the squared sub-Riemannian distance as well as the cut locus on generalized Heisenberg-type groups and on star graphs respectively.  Furthermore, the long-standing open problem of Gaveau-Brockett is completely solved in the case of $N_{3, 2}$, the free Carnot group of step two and 3 generators.

\medskip


\medskip


\medskip

\renewcommand{\theequation}{\thesection.\arabic{equation}}
\section{Introduction}\label{sec1}
\setcounter{equation}{0}

\medskip

In this article, we present some results on a fundamental problem of obtaining the exact formula for the sub-Riemannian distance on $2$-step nilpotent groups. It is in connection with the study of many other topics such as geodesics, abnormal extremals, cut locus, optimal transport, asymptotic estimates for the heat kernel, various harmonic analysis problems, etc. See for example \cite{G77}, \cite{A81}, \cite{S86}, \cite{HM89}, \cite{LS95}, \cite{GK95}, \cite{BGG00}, \cite{M02}, \cite{AR04}, \cite{R05}, \cite{Li07}, \cite{Li09}, \cite{Li10}, \cite{LLMV13}, \cite{MM17}, \cite{LZ19} and the references therein.

Recall that $2$-step groups are closely connected with simply connected homogeneous Riemannian manifolds with sectional curvature satisfying $-4 \leq K \leq -1$,  see \cite[Section~0]{E03} and the references therein.  Geometry of $2$-step groups with a left invariant Riemannian metric has been studied for example in \cite{K83}, \cite{E94}, \cite{E94A}, \cite{GM00} and the references therein, as well as their subsequent research.

Here, we will consider only sub-Riemannian geometry of $2$-step groups. A lot of ``qualitative results'' have been obtained in $2$-step groups also in  more general sub-Riemannian setting, such as smoothness of the sub-Riemannian distance on an open everywhere dense subset, Sard property of the abnormal set, etc.  We refer the reader to \cite{RT05}, \cite{A08},  \cite{LDMOPV16} and references therein for further details. In particular, we notice that the second layer in the stratification of Lie algebra and abnormal set are main research object of ``qualitative results'' cited above in the setting of $2$-step groups.

In the present paper, we focus mainly on a ``quantitative problem'' which becomes much more difficult. More precisely, we will study the exact expression of the Carnot-Carath\'eodory distance $d(g, g')$ on a $2$-step group $\G$. First, by the left-invariant property of $d$, it suffices to consider
\[
d(g) := d(o, g) \quad \mbox{where $o$ is the identity element of $\G$.}
\]
Second, recall that $d$ is continuous on $\G$ with respect to the usual topology (cf. \cite{VSC92}).  So it is enough to determine $d(g)$ for $g$ in a suitable dense set. Recall that \cite[Proposition~15]{R13} (see also \cite[pp.~384--385]{ABB20}) says that the cut locus of $o$, namely the set of points where the squared distance is not smooth, has Lebesgue measure zero. Hence, using the above ``qualitative results'', we do not care too much about the exact formulas for $d(g)$ in the case where $g$ belongs to the second layer, abnormal set, or the cut locus of $o$ which are negligible from an analytic point of view. Third, it is worthwhile to point out that in the study of Carnot-Carath\'eodory distance and of heat kernel on $2$-step group $\G$, by  basic properties of direct product, we may assume that $\G$ is irreducible, that is it cannot be decomposed into direct product of two non-trivial subgroups.  However, we will not assume in the sequel that $\G$ is irreducible.

The exact computation for $d$ on nilpotent Lie groups dates back to the work of Gaveau (cf. \cite{G77})  and Brockett (cf. \cite{B82}). To get the distance, the classical method is to find length-minimizing horizontal path parametrized with constant speed by using tools of control theory. See for example \cite{G77} and \cite{BGG00} for the case of general Heisenberg groups. In brief, for fixed $g \neq o$,
we first find all normal geodesics joining $o$ to $g$ via solving some ODE. Then among  all these normal geodesics, we compute the length of the/a shortest one. In other words, we have to find the best inverse image of the sub-Riemannian exponential mapping based at $o$ (see \eqref{nFE} below for the definition). It is usually extremely complicated and not practical.

To our best knowledge, a complete and exact calculation for  the sub-Riemannian  distance in the setting of stratified groups has been obtained only for general Heisenberg groups and Heisenberg-type groups, see \cite{G77}, \cite{BGG00}, \cite{TY04} and \cite{R05}. We will generalize these known results by introducing a new approach based on operator convexity,  which can be adapted to other situations and related problems. In particular, we can get some results about cut locus, shortest and normal geodesics, as well as short-time asymptotic behavior of the heat kernel. Moreover, our method allows us to improve the results in \cite[\S  3]{BGG96} and simplify their proof. Also, the method will be adapted in \cite{LZ20A} to obtain the asymptotic behavior at infinity with time $1$ of the heat kernel on a large class of step-two groups.

The paper is organized as follows. In Section \ref{s2} we give some preliminary material and the main results. In Section \ref{SFP} we collect fine properties of the function $-s \cot{s}$, in particular its operator convexity. In Section \ref{SP2}, we obtain a lower bound of the squared sub-Riemannian distance as well as the exact formulas of $d(g)^2$ under the assumption that $g \in \M$ defined by \eqref{m} below. The structure of $\M$ will be studied in Section \ref{S5s}. We consider (``good'') normal geodesics in Section \ref{PTHN4}. A slight generalization of our main result, as well as the geometric meaning of $\M^c$ on an enormous class of step-two groups, will be provided in Section \ref{Pthn6}. Sections \ref{ghg}, \ref{MG}, \ref{star} and \ref{N32} are devoted to the applications of our main theorems. We characterize the squared distance as well as the cut locus on generalized Heisenberg-type groups in Section \ref{ghg}, and on star graphs in Section \ref{star}. The squared distance will be considered on M\'etivier groups in Section \ref{MG}. Finally, in Sections \ref{N32}, we explain briefly how to determine the square sub-Riemannian distance for general step-two groups. As an example, we solve the Gaveau-Brockett optimal control problem on the free Carnot group of step two and $3$ generators.

Some of our results were presented at the $8^{\mbox{th}}$ ICCM  held at Tsinghua University in June 2019 (see the announcement
\cite{Li19}).

\medskip

\renewcommand{\theequation}{\thesection.\arabic{equation}}
\section{Notations and statement of results}\label{s2}
\setcounter{equation}{0}

\medskip

Let $q, m \in \N^*$ and let $\U = \left\{ U^{(1)}, \ldots, U^{(m)} \right\}$ be an $m$-tuple of linearly independent $q \times q$ skew-symmetric matrices with real entries. Then we have $m \leq \frac{q (q - 1)}{2}$. Recall (cf. for example \cite[\S  3.2]{BLU07}, also \cite{E03} for some equivalent definitions) that the step-two Carnot group of type $(q, m, \U)$, $\G = \G(q, m, \U)$, can be considered as $\R^q \times \R^m$ with the group law
\[
(x, t) \cdot (x', t') = (x + x', t + t' + 2^{-1} \langle \U x, x' \rangle),
\]
where $\langle \U x, x' \rangle$ denotes the $m$-tuple
\[
(\langle U^{(1)} x, x' \rangle, \ldots, \langle U^{(m)} x, x' \rangle) \in \R^m,
\]
and $\langle \cdot, \cdot \rangle$ (resp. $\langle \cdot, \cdot \rangle_{\C}$) stands for the inner product in $\R^q$ (resp. $\C^q$).

Let $U^{(j)} := (U^{(j)}_{l, k})_{1 \leq l, k \leq q}$ ($1 \leq j \leq m$). The vector fields
\[
\X_l := \frac{\partial}{\partial x_l} + \frac{1}{2} \sum_{j = 1}^m \Big( \sum_{k = 1}^{q} U^{(j)}_{l, k} x_k \Big) \frac{\partial}{\partial t_j}, \qquad 1 \leq l \leq q,
\]
are left invariant on $\G$. The canonical sub-Laplacian is
\[
\Delta := \Delta_{\U} = \sum_{l = 1}^{q} \mathrm{X}_l^2.
\]

Here we will restrict our attention to the canonical sub-Laplacian $\Delta_{\U}$. However, this assumption is not essential. Indeed, for an arbitrary sub-Laplacian on $\G(q, m, \U)$, namely
\[
\Delta_{\U, \mathbb{A}} := \sum_{l = 1}^q \left( \sum_{j = 1}^q a_{l, j} \X_j \right)^2
\]
where the real matrix $\mathbb{A} = (a_{l, j})_{1 \leq l, j \leq q}$ is non-singular, it is always possible to perform a linear change of basis so that $\Delta_{\U, \mathbb{A}}$ is turned into the canonical sub-Laplacian $\Delta_{\widetilde{\U}}$ of $\G(q, m, \widetilde{\U})$ for some $\widetilde{\U}$. See for example \cite{BLU07} for more details.

By slightly abusing of notation in the sequel, $0$ denotes the number $0$ or the origin in Euclidean space. Let $g = (x, t) \in \R^q \times \R^m$. Recall that $d$ denotes the Carnot-Carath\'eodory distance associated to $\Delta$, $o = (0, 0)$ is the identity element of $\G$ and $d(g) = d(g, o)$. The dilation on $\G$ is defined by
\[
\delta_r(x, t) := (r \, x, \  r^2 \, t), \qquad \forall r > 0, \, (x, t) \in \G.
\]
It is well-known that
\begin{align} \label{scaling}
d(\delta_r(g)) = r \, d(g), \quad \forall r > 0, \ g \in \G.
\end{align}
We also write
\begin{align*}
| x |^2 = \sum_{k =1}^{q} x_k^2,  \quad | t |^2 = \sum_{j = 1}^m t_j^2,  \quad t \cdot \lambda = \sum_{j = 1}^m t_j \lambda_j \  \mbox{\  for \ $\lambda = (\lambda_1, \ldots, \lambda_m) \in \R^m$}.
\end{align*}
The Haar measure $dg$ coincides with the $(q + m)$-dimensional Lebesgue measure on $\G$.

The heat semigroup $(e^{h \Delta})_{h
> 0}$ has a convolution kernel $p_h$, in the sense that
\begin{eqnarray*}
e^{h \Delta} u(g) = u \ast p_h(g) = \int_{\G} u(g') \, p_h(g'^{-1} g) \, dg'
\end{eqnarray*}
for suitable functions $u$. It is well known that
\begin{align} \label{sp}
p_h(x, t) = h^{-\frac{q}{2} - m} p_1(x/\sqrt{h}, t/h), \qquad h > 0, \ (x, t) \in \G.
\end{align}
Up to a positive constant, $p_1 = C p$ and (cf. \cite{C79} or \cite[Theorem 4, p.\,293]{BGG96} with a slight modification of notations, also \cite[Proposition 4]{MM16} and references therein)
\begin{align} \label{2c0}
p(x, t) := \int_{\R^m} V(\lambda) \,  e^{-\frac{\widetilde{\phi}((x, t); \lambda)}{4}} \, d\lambda,
\end{align}
where
\begin{gather}
\widetilde{U}(\lambda) := \sum_{j = 1}^m \lambda_j U^{(j)},  \quad  U(\lambda) := i \, \sum_{j = 1}^m \lambda_j U^{(j)},  \quad  V(\lambda) := \Big( \mathrm{det} \frac{U(\lambda)}{\sinh{U(\lambda)}} \Big)^{\frac{1}{2}},  \label{2c'}  \\
\widetilde{\phi}((x, t); \lambda) := \langle U(\lambda) \coth{U(\lambda)} \, x, x  \rangle - 4 \, i \, t \cdot \lambda,    \label{2c}
\end{gather}
and the meaning of the matrix-valued functions herein is obvious. Notice that Varadhan's formulas in this situation have been proved (cf. \cite{L87}, \cite{BA88}, \cite{VSC92}, \cite{AH05} and the references therein), namely
\begin{align} \label{VF}
\lim_{h \longrightarrow 0^+} 4 h \ln{p_h(g)} = -d(g)^2.
\end{align}

Recall that the usual norm of a $q \times q$ matrix $A$, considered as a linear operator on $\C^q$, is defined as
\[
\| A \| := \sup_{|z| = 1, z \in \C^q} |Az|.
\]
A function $u$, defined on a convex set $\Omega \subset \R^n$, is called concave if
\begin{align} \label{Dcp}
u(s a + (1-s)b) \geq s \, u(a) + (1 - s) \, u(b), \quad \forall 0 < s < 1, \ a \neq b \in \Omega.
\end{align}
If the relation \eqref{Dcp} holds with ``$>$'' instead of ``$\geq$'' we say that $u$ is strictly concave. And $u$ is strongly concave if there exists a constant $C > 0$ such that
\[
u(s a + (1-s)b) \geq s u(a) + (1 - s) u(b) + C \, s \, (1 - s) \, |a - b|^2, \quad \forall 0 \leq s \leq 1, \ a, b \in \Omega.
\]
Similarly, we say that $u$ is (strictly, strongly resp.) convex if $-u$ is (strictly, strongly resp.) concave.

In the sequel, let us introduce the initial reference set
\begin{align} \label{oa}
\OA := \left\{ \tau \in \R^m; \ \max_{|x| = 1} \langle U(\tau)^2 \, x, x \rangle  < \pi^2 \right\} = \left\{ \tau \in \R^m; \ \| U(\tau) \| < \pi \right\},
\end{align}
which is obviously an absorbing, convex and bounded open set.

\medskip

\subsection{Basic results}

\medskip

The starting point of this article is the following:

\begin{prop} \label{TH1}
For any $g = (x, t) \in \G$, the reference function defined by
\begin{align} \label{2cn}
\phi(g; \tau) := \widetilde{\phi}(g; i \tau) = \langle U(\tau) \cot{U(\tau)} \, x, x  \rangle + 4 t \cdot \tau, \quad \tau \in \OA,
\end{align}
is smooth and concave.
\end{prop}

\begin{remark} \label{nRk21}
(1) It is useful to consider the closure of $\phi(g; \cdot)$,  on $\overline{\OA}$,  defined by (see for example \cite[Section 7]{R70})
\[
\mathrm{cl} \phi(g; \tau) := \limsup_{\tau' \longrightarrow \tau}  \phi(g; \tau'),
\]
which is a closed, upper semicontinuous and proper concave function on  $\overline{\OA}$. It follows from \cite[Section 7]{R70} that $\mathrm{cl} \phi(g; \cdot)$ coincides with its usual extension on $\overline{\OA}$, that is, $\mathrm{cl} \phi(g; \tau) = - \infty$ if $\tau \in \partial \OA$ and the orthogonal projection of $x$ on $\pi^2$-eigenspace of $U^2(\tau)$ is nonzero, and $\mathrm{cl} \phi(g; \tau) = \phi(g; \tau)$ otherwise with the understanding $0 \cdot \infty = 0$. Notice that $\mathrm{cl} \phi(g; \cdot)$ attains its maximum in $\overline{\OA}$.

(2) Also notice that the function $\phi(g; \cdot)$ is well-defined provided the spectrum of $U(\tau)$ does not contain any $k \pi$ ($k \in \Z \setminus \{ 0 \}$).
\end{remark}

In some concrete situations, the above result can be improved to ``strongly concave for suitable $g$''. For a concave function, we are interested in its maximum or supremum. Moreover, motivated by the results obtained in \cite{BGG96} and \cite{BGG00}, we ask naturally if there is any connection between $d^2(g)$ and $\sup_{\tau \in \OA} \phi(g; \tau)$. We have the following bound:

\begin{theo} \label{MP1}
It holds that $d(g)^2 \geq \sup_{\tau \in \OA} \phi(g; \tau)$ for any $g \in \G$.
\end{theo}

We will show that if $\phi(g; \cdot)$ attains its maximum at $\theta \in \OA$, then $d(g)^2 = \phi(g; \theta)$ (see Theorem \ref{THN5} below). However we first provide a slightly weaker result which will allow us to characterize, in an  enormous class of step-two groups, the cut locus of $o$,
\begin{align} \label{DCUT}
\mathrm{Cut}_o := \mathcal{S}^c, \qquad \mathcal{S} := \{g; \, \mbox{$d^2$ is $C^{\infty}$ in a neighborhood of $g$}\}
\end{align}
(cf. Theorem \ref{aT121} below),  and understand Theorem \ref{THN4} below better. Moreover we can obtain, in its proof, some results about short-time asymptotic behavior, as well as asymptotic behavior at infinity with time $1$, of the heat kernel.

To begin with, let $\nabla_{\tau} = (\frac{\partial}{\partial \tau_1}, \ldots, \frac{\partial}{\partial \tau_m})$ denote the usual gradient on $\R^m$. Inspired by the method of stationary phase, we set
\begin{align} \label{m}
\M &:= \left\{ (x, t); \exists \, \theta \in \OA \mbox{ s.t. } t = -\frac{1}{4} \nabla_{\theta} \langle U(\theta) \cot{U(\theta)} \, x, x  \rangle, \ \mathrm{det} \left(- \mathrm{Hess}  \, \phi(g; \theta) \right) \neq 0 \right\} \nonumber \\
&=  \left\{ (x, t); \exists \, \theta \in \OA \mbox{ s.t. $\theta$ is a nondegenerate critical point of $\phi(g; \cdot)$ in $\OA$}  \right\}.
\end{align}

Obviously, $\M \subset \G$ is symmetric and scaling invariant; namely, if $g \in \M$, then we have $g^{-1} = -g \in \M$ and $\delta_r(g) \in \M$ for all $r > 0$. We can show that for $g \in \M$, the $\theta$ defined by \eqref{m} is unique, see Theorem \ref{TH3N} or a slight improvement in Proposition \ref{nP65} below. Notice that $\M \neq \emptyset$, this is exactly \cite[Proposition 6]{MM16} which plays a crucial role therein. More precisely, using \cite[Lemma 10]{MM16}, \cite[Proposition 6]{MM16} can be reformulated as Proposition \ref{nPP510} below and will be proved by a new method. However, from a sub-Riemannian geometric point of view, the fact that $\M \neq \emptyset$ is clear, see for example \cite[Remark 4]{LZ20}.

By \cite[IMPORTANT, p.146]{NP18},
$\theta$ is a critical point of $\phi(g; \cdot)$ iff. $\theta$ is its global maximizer. Then we have the following characterisation of the squared distance.

\begin{theo} \label{TH2}
Let $g_0 = (x_0, t_0) \in \M$ and $\theta_0$ is the unique nondegenerate critical point of $\phi(g_0; \cdot)$ in $\OA$. Then
\begin{align} \label{csd}
d(g_0)^2  = \phi(g_0; \theta_0) = \sup_{\tau \in \OA}  \phi(g_0; \tau).
\end{align}
\end{theo}

\begin{remark}
The implicit function theorem implies that $\M \subset \G$ is an open set and $d^2 \in C^{\infty}(\M)$. Furthermore, if $\overline{\M} = \G$, such as generalized Heisenberg-type groups and star graphs studied in Sections \ref{ghg} and \ref{star} below, it follows from Theorem \ref{aT121} below that $\M^c$ is exactly the cut locus of $o$. In such cases, the change of variables defined via (\ref{m}) has been used for example in \cite{Li06}.
\end{remark}

\medskip

\subsection{Structure of $\M$}

\medskip

Now we discuss in more detail the set $\M$. To begin with, let us set
\begin{align} \label{nOM}
\OM := \{g; \, \phi(g; \cdot) \mbox{  attains its maximum in } \OA \} := \OM_1 \cup \OM_2,
\end{align}
with
\begin{align}
\OM_1 &:= \{g; \, \mbox{the set of  global maximizers of $\phi(g; \cdot)$ in $\OA$ is a singleton} \}, \\
\OM_2 &:= \{g; \, \mbox{the set of  global maximizers of $\phi(g; \cdot)$ in $\OA$ has at least two points} \}. \label{OM2}
\end{align}

Notice that $o \in \OM_2$ and
\[
\OM \setminus\! \{ o \} = \left\{ \left(x, -\frac{1}{4} \nabla_{\theta} \langle U(\theta) \cot{U(\theta)} \, x, \, x \rangle \right); \, x \in \R^q \setminus\! \{ 0 \}, \theta \in \OA  \right\}.
\]

Since the map
\begin{align*}
(\R^q \setminus\! \{0\}) \times \OA &\longrightarrow (\R^q \setminus\! \{0\}) \times \R^m \subset \R^{q + m} \\
(x, \tau)  &\longmapsto \left( x, -\frac{1}{4} \nabla_{\tau} \langle U(\tau) \cot{U(\tau)} \, x, \, x \rangle \right)
\end{align*}
is smooth, it follows from Proposition \ref{nP65} below and Morse-Sard-Federer Theorem (cf.  \cite[p. 72]{KP02}) that $\OM_2$ has measure zero in $\G$. More properties of elements in $\OM_2$ will be found in Corollary \ref{nCCn} and Section \ref{S5s} below. Remark that $\M \subset \OM$. Moreover,

\begin{theo} \label{TH3N}
It holds that $\M = \OM_1$.
\end{theo}

\begin{remark}
The result says that if $\theta$ is the unique critical point of the concave smooth function $\phi(g; \cdot)$ in $\OA$, then it has to be nondegenerate. It is interesting to notice that this property is obviously valid under the assumption of ``strongly concave''. However, for a smooth function $u$, it is usually no longer true if we only assume ``concave'', even ``strictly concave''; a typical counter-example is the strictly concave and smooth function on $\R$, $u(s) = - s^4$ with its unique critical point $0$.
\end{remark}

\medskip

\subsection{Characterization of (``good'') normal geodesics} \label{nSSn23}

\medskip

From now on, we will study shortest and normal geodesics from $o$ to any given $g \neq o$. For this purpose, we start by recalling some elementary notations or facts in the setting of $2$-step groups.
We refer the reader to
\cite{K83}, \cite{GK95}, \cite{LS95}, \cite{M02}, \cite{R14}, \cite{ABB20} and the references therein for more details:
\begin{enumerate}[(I)]
  \item  A horizontal path is said to be a shortest geodesic if it realizes the Carnot-Carath\'eodory distance between its extremities.

 \item It follows from \cite[\S 2.4]{R14} that all shortest geodesics are normal.

\item Normal geodesics are projections on $\G$ of normal Pontryagin extremals, namely the integral curves of the sub-Riemannian Hamiltonian on the cotangent bundle $T^* \G$. The equation of normal Pontryagin extremals $(\gamma(s), (\xi(s), \tau(s)))$ ($0 \le s \le 1$) and normal geodesics \textit{that begin at the identity} $\gamma(s)$ ($0 \le s \le 1$) can be found in many articles. More precisely, the normal geodesic with initial covector $(\zeta(0), 2 \, \theta_0) \in \R^q \times \R^m$,
\[
\gamma(s) := \gamma(\zeta(0), 2 \, \theta_0; s) = (x(s), t(s)): \  [0, \ 1] \longrightarrow \G,
\]
is defined by
\begin{gather}
x(s) := \int_0^s \zeta(r) \, dr,  \label{GEn1*} \\
t(s) := \frac{1}{2} \int_0^s \langle \U x(r), \, \zeta(r) \rangle \, dr, \label{GEn2}
\end{gather}
with
\begin{align} \label{GEn3*}
\zeta(s) :=  e^{2 s \widetilde{U}(\theta_0)} \, \zeta(0).
\end{align}

The normal extremal lift of $\gamma$ is $(\gamma(s), (\xi(s), \tau(s)))$ ($0 \le s \le 1$) where
\begin{align} \label{HaEP}
\tau(s) \equiv 2 \theta_0, \qquad \xi(s) := \zeta(s) - \frac{1}{2} \widetilde{U}(\tau(s)) x(s) = \zeta(s) - \widetilde{U}(\theta_0) x(s).
\end{align}
And $(\xi(0), \tau(0)) = (\zeta(0), 2 \theta_0)$ (resp. $(\xi(1), \tau(1)) = (\xi(1), 2 \theta_0)$) is called the initial covector (resp. final covector) of $(\gamma(s), (\xi(s), \tau(s)))$.
Furthermore, the length of $\gamma$, $\ell(\gamma)$, equals $|\zeta(s)| = |\zeta(0)|$. See for example \cite{RT05} (or \cite[\S 7.5.1]{ABB20} with slight modification) for more details. Also notice that the normal extremal lift of $\gamma$ is defined as $(\gamma(s), (\xi(s), \tau(s)), -\frac{1}{2}, \zeta(s))$ ($0 \le s \le 1$) in \cite{RT05}.

Recall that a normal geodesic $\gamma$ is said to be \textit{abnormal} (i.e. singular) if it has two distinct normal extremal lifts (see \cite[Remark 8]{RT05} and \cite[Remark 2.4]{R14}). And $\gamma$ is called \textit{strictly normal} if it is not abnormal.

\item The sub-Riemannian exponential map from $o$ is defined by
\begin{align} \label{nFE}
\exp: \  \R^q \times \R^m &\longrightarrow \R^q \times \R^m \nonumber \\
(\zeta(0), 2 \, \theta_0) &\longmapsto \gamma(\zeta(0), 2 \, \theta_0; 1) = (x(1), t(1)).
\end{align}
Notice that it is smooth. Also it is surjective. See for example \cite[Corollary 12.14 and Proposition 8.38]{ABB20}, or \cite[Theorem 2.14]{R14} for a slightly weaker but much more general result. Recall that
\begin{align} \label{nFG}
\gamma(\zeta(0), 2 \, \theta_0; s) = \exp\left\{ s \, (\zeta(0), 2 \, \theta_0) \right\}, \qquad 0 \le s \le 1.
\end{align}

\item By \cite[\S 11.1]{ABB20} (cf. also \cite{RT05} and \cite{R14}), for any $g \in \mathrm{Cut}_o^c = \mathcal{S}$ (cf. \eqref{DCUT}), there exist a unique shortest geodesic $\gamma_g$ joining $o$ to $g$ which is not abnormal, that is $\gamma_g$ admits a unique normal extremal lift. Moreover, the following result is a simplified version of \cite[proof of Theorem 11.8]{ABB20}:

\begin{lem} \label{AxL}
Assume that $g_0 = \exp{(\zeta_0, 2 \theta_0)} \in \mathcal{S}$ and the unique shortest geodesic from $o$ to $g_0$ is $\gamma_{g_0}(s) = \exp\{s \, (\zeta_0, 2 \theta_0)\}$ ($0 \le s \le 1$). Then, there exists a neighborhood  of $(\zeta_0, 2 \theta_0)$,  $V_{(\zeta_0, 2 \theta_0)} \subset \R^q \times \R^m$, such that the sub-Riemannian exponential map is a diffeomorphism from $V_{(\zeta_0, 2 \theta_0)}$ onto its image $O_{g_0} := \exp{(V_{(\zeta_0, 2 \theta_0)})} \subset \mathcal{S}$. Furthermore, for any $g' = \exp{(\zeta, 2 \theta)}$ with $(\zeta, 2 \theta) \in V_{(\zeta_0, 2 \theta_0)}$, the unique shortest geodesic from $o$ to $g'$ is $\gamma_{g'} = \exp{(s \, (\zeta, 2 \theta))}$ ($0 \leq s \leq 1$), which is not abnormal. Also, $d(g')^2 = |\zeta|^2$.
\end{lem}

\item We have the following simple observation:

\begin{lem} \label{AxL1}
Suppose that $g^{(j)} = \exp\{(\zeta^{(j)}, \, 2 \theta^{(j)})\}$ ($j \in \N^*$) and $\exp\{s \, (\zeta^{(j)}, \, 2 \theta^{(j)})\}$ ($0 \le s \le 1$) is a shortest geodesic from $o$ to $g^{(j)}$. Assume that $(\zeta^{(j)}, \, 2 \theta^{(j)}) \longrightarrow (\zeta^{(0)}, \, 2 \theta^{(0)})$ as $j \longrightarrow +\infty$. Then $\exp\{s \, (\zeta^{(0)}, \, 2 \theta^{(0)})\}$  ($0 \le s \le 1$) is a shortest geodesic from $o$ to $g^{(0)} = \exp\{(\zeta^{(0)}, \, 2 \theta^{(0)})\}$.
\end{lem}
\end{enumerate}

However, a fundamental and much more interesting problem is to determine all normal geodesics as well as the shortest one(s) \textit{from $o$ to any given $g \neq o$}. There is a great lack of known result about this basic question. A complete description is only known for general Heisenberg groups (see \cite{BGG00}).

Theorems \ref{TH2} and \ref{MP1} allow us to describe the shortest geodesic between $o$ and $g_0 = (x_0, t_0) \in \M$, $\gamma_{g_0}$. Indeed, applying \cite[Proposition 2]{RT05} to the function $\phi(g; \theta_0)$, there exists a unique shortest geodesic between $o$ and $g_0$ which admits a normal extremal lift with final covector $(\xi(1), \tau(1)) = (U(\theta_0) \cot{U(\theta_0)} \, x_0, 2 \theta_0)$. Inspired by this observation as well as the results obtained in \cite{BGG00} in the setting of generalized Heisenberg groups, we adapt our method to find all normal geodesics joining $o$ to any given $g$ except a set of measure zero.

To begin with, we divide all normal geodesics into two cases: the ``good'' ones and the ``bad'' ones. More precisely, let $\mathbb{I}_q$ denote the $q \times q$ identity matrix and set
\begin{align} \label{nD1V}
\mathcal{V} := \left\{ \vartheta \in \R^m; \,  \mathrm{det} (k \, \pi \, \mathbb{I}_q - U(\vartheta)) \neq 0, \quad \forall k \in \N^* \right\},
\end{align}
which is obviously an open subset in $\R^m$. Notice that its complement, $\mathcal{V}^c$, is a countable union of algebraic sets $\left\{ \vartheta \in \R^m; \,  \mathrm{det} (k \, \pi \, \mathbb{I}_q - U(\vartheta)) = 0 \right\}$ ($k \in \N^*$). Thus $\mathcal{V}^c$ has measure zero and $\mathcal{V} \subset \R^m$ is dense.

The normal geodesic $\gamma(\zeta(0), 2 \, \theta_0; s)$ ($0 \le s \le 1$) is called \textit{good} if $\theta_0 \in \mathcal{V}$ and \textit{bad} if $\theta_0 \in \mathcal{V}^c$. Notice that even a nontrivial normal geodesic $\gamma$ can be both ``good'' and ``bad'', since there may exist two different parameters $(\zeta, \theta) \in \left( \R^q \setminus \{ 0 \} \right) \times \mathcal{V}$ and $(\zeta', \theta') \in \left( \R^q \setminus \{ 0 \} \right) \times \mathcal{V}^c$ such
that $\gamma(s) = \gamma(\zeta, 2 \, \theta; s) = \gamma(\zeta', 2 \, \theta'; s)$. In such case, we can get more information about $\gamma$ and its endpoint $\gamma(1)$, see for example \cite[Remark~8]{RT05},  \cite[Remark~2.4]{R14} and \cite{LZ20} for more details.

Remark that the set of the endpoints of ``bad'' normal geodesics,
\begin{align} \label{nDoW}
\mathcal{W} := \left\{\gamma(\zeta(0), 2 \, \theta_0; 1) = (x(1), t(1)); \ \zeta(0) \in \R^q, \ \theta_0 \in \mathcal{V}^c \right\},
\end{align}
is negligible. Indeed, it is the image under a smooth map, the sub-Riemannian exponential map, of
\[
\{(\zeta(0), 2 \, \theta); \ \zeta(0) \in \R^q, \ \theta \in \mathcal{V}^c\}
\]
which has measure zero. Thus we focus mainly on the ``good'' normal geodesics. Let us begin by providing more information about them.

Assume that
\[
\gamma(s) := \gamma(\zeta(0), 2 \, \theta_0; s) = (x(s), t(s)), \qquad 0 \le s \le 1,
\]
is a non-trivial ``good'' normal geodesic, namely $\zeta(0) \neq 0$ and $\theta_0 \in \mathcal{V}$. Using the Spectral Theorem (in $\C^q$ but not in $\R^q$), a direct calculation shows that
\begin{align} \label{iniC}
x_0 := x(1) \neq 0, \qquad \zeta(0) = \frac{U(\theta_0)}{\sin{U(\theta_0)}} e^{- \widetilde{U}(\theta_0)} \, x_0.
\end{align}
Substituting this in \eqref{GEn3*} and \eqref{GEn1*}, a simple calculation implies that
\begin{align} \label{GEn1}
x(s) = \frac{\sin{(s \, U(\theta_0))}}{\sin{U(\theta_0)}} e^{(s - 1) \widetilde{U}(\theta_0)} \, x_0,
 \qquad
t(s) = \frac{1}{2} \int_0^s \langle \U x(r), \, \zeta(r) \rangle \, dr,
\end{align}
where we set by convention $\frac{\sin{(s \cdot 0)}}{\sin{0}} = s$ and
\begin{align}
\zeta(s) = \frac{U(\theta_0)}{\sin{U(\theta_0)}} e^{(2 s - 1) \widetilde{U}(\theta_0)} \, x_0.   \label{GEn3}
\end{align}
Also notice that \eqref{GEn3} and the first equality in \eqref{GEn1} could be considered as a counterpart of \cite[(2.14)]{BGG96}.

We characterize all ``good'' normal geodesics as the following result.

\begin{theo} \label{THN4}
Assume that $x_0 \neq 0$ and $\theta_0 \in \mathcal{V}$. If $\theta_0$ is a critical point of $\phi(g_0; \cdot)$ for some $g_0 = (x_0, t_0)$, then the ``good'' normal geodesic $\gamma$, defined by \eqref{GEn1} and \eqref{GEn3}, steers $o$ to $g_0$. In particular, we have
\begin{align} \label{UED}
d(g_0)^2 \leq \ell(\gamma)^2 = \left| \frac{U(\theta_0)}{\sin{U(\theta_0)}} x_0 \right|^2 = \phi(g_0; \theta_0) .
\end{align}
Conversely, let $t_0 = t(1)$ and $g_0 = (x_0, t_0)$, then $\theta_0$ is a critical point of $\phi(g_0; \cdot)$.
\end{theo}

\begin{remark} \label{nRk26}
(1) The second equality in \eqref{UED} could be considered as a counterpart of \cite[(2.14) and (2.15)]{BGG96}. We will provide a direct proof which is of independent interest, see Proposition \ref{NC51} below.

(2) We notice that by Morse-Sard-Federer Theorem, the set
\begin{align} \label{nDoN}
\mathcal{N} := \left\{g = (x, t); \ x \neq 0, \mbox{$\exists \theta \in \mathcal{V}$ is a degenerate critical point of $\phi(g; \cdot)$} \right\}
\end{align}
has measure zero in $\G$.

(3) Under the additional assumption that $q > m$, by \cite[Proposition 7]{AGL15}, $\phi(g; \cdot)$ has only a finite number of critical points except some subset of measure zero.

(4) In the setting of generalized Heisenberg groups and Heisenberg-type groups, the equation of geodesics becomes very concise via their special group structure. See \cite{G77}, \cite{BGG00}, \cite{TY04} and \cite{R05} for more details. In general, this is very complicated, but we should simplify a little the second equation in \eqref{GEn1}, that is \eqref{GEn2}, by the Spectral
Theorem (see the proof of \eqref{EFT} below).
\end{remark}

As a consequence, we obtain the following qualitative result:

\begin{cor} \label{nCx}
Let $g = (x, t) \in \mathcal{W}^c$. Then
\[
d^2(g) = \inf\left\{ \phi(g; \theta) = \left| \frac{U(\theta)}{\sin{U(\theta)}} x \right|^2; \  \mbox{$\theta$ is a critical point of $\phi(g; \cdot)$ in $\mathcal{V}$} \right\}.
\]
\end{cor}

\medskip

\subsection{Characterization of $d(g_0)^2$ for $g_0 \in \widetilde{\M} \setminus \{ o \}$ as well as the shortest geodesic from $o$ to such $g_0$}

\medskip

Recall that (cf. \cite[\S 2.1]{RT05} or \cite{CLSW98}) the \textit{proximal sub-differential} of a continuous function on $\G \cong \R^q \times \R^m$, $u: \R^q \times \R^m \longrightarrow \R$, at $g$ is defined by
\begin{align} \label{PSD}
\partial_P u(g) := \{dw(g); \, &\mbox{$w$ is smooth in a neighborhood of $g$} \nonumber \\
\mbox{} & \mbox{and $u - w$ attains a local minimum at $g$} \}.
\end{align}

Let $g_0 = (x_0, t_0) \in \widetilde{\M} \setminus \{o\}$ and $\theta_0 \in \OA$ is a global maximizer of $\phi(g_0; \cdot)$. Consider the $C^{\infty}$ function
$\phi_{g_0}(g) := \phi(g; \theta_0)$.
By Theorems \ref{THN4} and \ref{MP1}, we have
\[
d(g_0)^2 = \phi(g_0; \theta_0) = \phi_{g_0}(g_0), \qquad \phi_{g_0}(g) \leq d(g)^2, \ \forall g.
\]

Then the vector $d\phi_{g_0}(g_0) = 2 \, (U(\theta_0) \cot{U(\theta_0)} \, x_0, \ 2 \theta_0)$ belongs to the proximal sub-differential of $d^2$ at the point $g_0$. Notice that for $g_0 \in \OM_2$, $\partial_P d^2(g_0)$ has at least two distinct elements. In conclusion, by \cite[Proposition~2 and Remark~8]{RT05}, we get the following:

\begin{theo} \label{THN5}
Suppose that $g_0 = (x_0, t_0) \in \widetilde{\M} \setminus \{o\}$ and $\theta_0 \in \OA$ is a global maximizer of $\phi(g_0; \cdot)$.
Then $d(g_0)^2 = \phi(g_0; \theta_0)$ and there exists a unique shortest geodesic joining $o$ to $g_0$, which is defined by \eqref{GEn1} and \eqref{GEn3}. In addition, the shortest geodesic is also abnormal (i.e. singular) if $g_0 \in \OM_2$.
\end{theo}

\begin{remark} \label{nRKn25}
(1) For example, we have
\[
d(x, 0)^2 = |x|^2 =  \phi((x, 0); 0) = \sup_{\tau \in \OA} \phi((x, 0); \tau), \qquad x \in \R^q,
\]
which can be found in \cite[p. 2104]{LM14}.

(2) Let $\zeta(0) \in \R^q \setminus \{ 0 \}$ and $\theta_0 \in \OA$. Then $g_0 = \exp{(\zeta(0), 2 \theta_0)} \in \widetilde{\M} \setminus \{o\}$, and the unique shortest geodesic from $o$ to $g_0$ is $\exp\{s \, (\zeta(0), 2 \theta_0)\}$ ($0 \le s \le 1$).
\end{remark}

Let us set
\begin{align} \label{Dabn}
\mathrm{Abn}_o^* = \left\{ g; \ \mbox{$\exists$ abonormal \textit{shortest} geodesic joining $o$ to $g$} \right\} \ni o,
\end{align}
which is included in
\[
\mathrm{Abn}_o = \left\{ g; \ \mbox{$\exists$ abonormal-normal geodesic joining $o$ to $g$} \right\}
\]
introduced in \cite{LDMOPV16}. Recall that $\mathrm{Abn}_o = \{ o \}$ iff. $\G$ is of M\'etivier (see section \ref{MG} below for the definition).
Combining Theorem \ref{THN5} with \cite[Proposition 11.4]{ABB20}, we get immediately:

\begin{cor} \label{nCCn}
It holds that $\OM_2 \subseteq \mathrm{Abn}_o^* \subseteq \mathrm{Cut}_o$.
\end{cor}

From (b) of Proposition \ref{nP65} below, we have that $\OM_2 = \{ o \}$ iff. $\G$ is of M\'etivier. Furthermore, it is easy to describe formally $\mathrm{Abn}_o^*$ and $\mathrm{Abn}_o$ by $\OM_2$. See \cite{LZ20} for more details.

\medskip

Moreover, we have the following:

\begin{theo} \label{THN6}
For any $g \in \overline{\widetilde{\M}}$, we have
\[
d(g)^2 = \sup_{\tau \in \OA}  \phi(g; \tau).
\]
\end{theo}

\begin{remark} \label{Nr27}
(1) Let $g \in \overline{\widetilde{\M}} \setminus \widetilde{\M}$. By the compactness of $\overline{\OA}$, from \cite[Proposition 4]{RT05} or Lemma \ref{AxL1}, there exist $\theta \in \partial \OA$ and $\zeta(0) \in \R^q$ satisfying $d(g)^2 = |\zeta(0)|^2$ such that the ``bad'' normal geodesic $\exp\{s \, (\zeta(0), 2 \, \theta)\}$ ($0 \le s \le 1$) is a shortest one from $o$ to $g$.

(2) In the case where $\overline{\M} = \G$ so $\overline{\widetilde{\M}} = \G$, there may exist infinitely many shortest geodesics from $o$ to any given $g \in \widetilde{\M}^c$. See \cite{BGG00} for the specific case of generalized Heisenberg groups, also \cite{R05} for Heisenberg-type groups.

(3) Indeed, we have $\overline{\M} = \overline{\widetilde{\M}}$, see \cite{LZ20} for more details.
\end{remark}

\medskip

\subsection{Basic properties of $2$-step groups satisfying $\overline{\M} = \G$}

\medskip

Recall that one of the important motivations of obtaining the exact formula for $d$ is to study asymptotic estimates for the heat kernel on the underlying group. However, there exists an essential difficulty in the case where $\overline{\M} \subsetneqq \G$. A very interesting problem is to characterize $2$-step groups satisfying $\overline{\M} = \G$, that will be called \textit{GM-groups} or groups \textit{of type GM}. Since the main goal of this article is to study the sub-Riemannian distance on general step-two groups, we leave this task to be completed in \cite{LZ20}. Here we only give a few basic properties of GM-groups that we will use. First of all, motivated by the case of general Heisenberg groups studied in \cite{BGG00}, we have the following sufficient condition:

\medskip

\begin{prop} \label{nPr210}
Let $x \in \R^q \setminus\! \{0\}$. Assume that for any $\tau \in \partial \OA$,  the orthogonal projection of $x$ on $\pi^2$-eigenspace of $U^2(\tau)$ is nonzero. Then we have $(x, t) \in \M$ for all $t \in \R^m$. Moreover, if the set of such $x$ is dense in $\R^q$, then $\overline{\M} = \G$.
\end{prop}

It follows from Proposition \ref{nPr210} that GM groups form a wild set. More precisely, for any given $\G(q, m, \U)$, we can construct an uncountable number of GM-groups $\G(q + 2 n, m, \widetilde{\U})$. See Subsection \ref{sNss81} below for more details.

It is well-known that the distance and the cut locus of a Riemannian manifold cannot in general be explicitly computed. Also, as we can find in this work, the expression of $d^2(g)$ on a step-two group is generally extremely complicated. However, GM-groups have consummate sub-Riemannian properties. A direct consequence of Theorem \ref{THN6} is the following:

\begin{cor}
Suppose that $\overline{\M} = \G$. Then  $d(g)^2 = \sup_{\tau \in \OA} \phi(g; \tau)$ for any $g \in \G$.
\end{cor}

Furthermore, recall that $\M$ is open and $\M \subset \mathrm{Cut}_o^c = \mathcal{S}$, we have the following:

\begin{theo} \label{aT121}
If $\overline{\M} = \G$, then the cut locus of $o$ is exactly $\M^c = \partial \M$.
\end{theo}

\begin{remark}
Indeed, the following properties are equivalent: (i) $\overline{\M} = \G$; (ii) $\mathrm{Cut}_o = \partial \M$; and (iii) $d(g)^2 = \sup_{\tau \in \OA} \phi(g; \tau)$ for any $g \in \G$. We refer the reader to \cite{LZ20} for more details as well as other equivalent characterizations of GM-groups, especially the geometric meaning of $\OM^c$.
\end{remark}

In this article, we will only describe the squared sub-Riemannian distance as well as the cut locus on two representative type of GM-groups: generalized Heisenberg-type groups as well as star graphs. More examples can be found in \cite{LZ20}.

\medskip

Next, we provide

\medskip

\subsection{Some improvements}

\medskip

Let us begin by improving Corollary \ref{nCx}. It follows from \cite[Corollary~12.14]{ABB20} that the squared sub-Riemannian distance in the setting of step-two groups is locally Lipschitz w.r.t. the usual Euclidean distance. So $d^2$ is Lipschitz on
\[
\mathcal{C}_r := \{(x, t); \, |x| < r, \ |t| < r^2\}, \qquad 0 < r \leq 1.
\]
Then, it deduces from the scaling property (see \eqref{scaling}) that the set defined formally by
\begin{align} \label{Rset}
\F &:= \left\{ \theta = \frac{1}{4} \nabla_t d(g)^2; \, g \in \mathcal{C}_1, \ g = (x, t) \not\in \mathrm{Cut}_o \right\} \nonumber \\
&= \left\{ \theta = \frac{1}{4} \nabla_t d(g)^2; \, g = (x, t) \not\in \mathrm{Cut}_o \right\} \subset \R^m
\end{align}
is bounded. It is also open from Lemma \ref{AxL}.

Thus, the global reference set
\begin{align} \label{GRs}
\FS := \overline{\F}
\end{align}
is compact in $\R^m$. Combining this with Corollary \ref{nCx} and \cite[Proposition 4]{RT05} (or Lemma \ref{AxL1}), we yield immediately

\begin{cor} \label{nC121}
We have for $g = (x, t) \in \mathcal{W}^c$,
\[
d^2(g) = \inf\left\{ \phi(g; \theta) = \left| \frac{U(\theta)}{\sin{U(\theta)}} x \right|^2; \  \mbox{$\theta$ is a critical point of $\phi(g; \cdot)$ in $\mathcal{V} \cap \FS$} \right\}.
\]
Moreover, for any $g \neq o$, there exists a shortest geodesic joining $o$ to $g$, $\exp\{s \, (\zeta_g(0), 2 \, \theta_g)\}$ ($0 \le s \le 1$), satisfies $\theta_g \in \FS$. Also, for any $\theta \in \FS$ and $r > 0$, there is $|\zeta_r| = r$ such that $\exp\{s \, (\zeta_r, 2 \, \theta)\}$ ($0 \le s \le 1$) is a shortest geodesic from $o$ to $\exp\{(\zeta_r, 2 \, \theta)\}$.
\end{cor}

To show the last claim, we can use the following simple observation: if $g_0 = \exp\{(\zeta(0), \, 2 \, \theta_0)\} \in \mathcal{S}$, then $\delta_r(g_0) = \exp\{(r \, \zeta(0), \, 2 \, \theta_0)\} \in \mathcal{S}$ for any $r > 0$ because of \eqref{scaling}.

Next, we will provide the formal expression of $d(g)^2$ for $g \not\in \overline{\OM} \subsetneqq \G$. Recall that $\mathrm{Cut}_o \cup \mathcal{W} \cup \mathcal{N}$ has measure zero (see \eqref{nDoW} and \eqref{nDoN}),
combining Corollary \ref{nC121} with Theorem \ref{THN4} and Lemma \ref{AxL},
it follows from the implicit function theorem that:

\begin{cor} \label{nC124}
Assume that $x_0 \neq 0$, $\theta_0 \in \mathcal{V} \cap \F$,
\[
g_0 = \left(x_0, -\frac{1}{4} \nabla_{\theta = \theta_0} \langle U(\theta) \cot{U(\theta)} \, x_0, \, x_0 \rangle \right) \not\in \mathrm{Cut}_o, \quad \mbox{ and } d(g_0)^2 = \phi(g_0; \theta_0).
\]
Moreover, we suppose that the Hessian matrix $\mathrm{Hess}_{\theta = \theta_0} \langle U(\theta) \cot{U(\theta)} \, x_0, \, x_0 \rangle$ is non-singular. Then we have in a neighborhood of $(x_0, \theta_0) \in \R^q \times \R^m$
\[
d^2\left(x, -\frac{1}{4} \nabla_{\theta} \langle U(\theta) \cot{U(\theta)} \, x, \, x \rangle \right) = \left| \frac{U(\theta)}{\sin{U(\theta)}} x \right|^2.
\]
\end{cor}

\medskip

\subsection{Afterwords}

\medskip

In this paper, we consider only Sub-Riemannian geometry related to the sub-Laplacian, such as the exact formula of the Carnot-Carath\'eodry distance, (shortest) geodesic from $o$ to a given point, and the cut locus of $o$. However, a straightforward modification allows us to obtain the corresponding results of Riemannian geometry related to the full Laplacian. Obviously, there are significant differences between the Sub-Riemannian geometry and the Riemannian geometry, and the latter is less difficult than the former since the corresponding reference function always is strongly concave in the initial reference set. Also notice that the scaling property (see \eqref{scaling}) is no longer valid for the Riemannian geometry.

Also remark that the main idea and
method can be adapted to other situations.

\medskip

\renewcommand{\theequation}{\thesection.\arabic{equation}}
\section{Fine properties of the function $-s \cot{s}$ and its derivatives}\label{SFP}
\setcounter{equation}{0}

\medskip

For $-\pi < s < \pi$,  define
\begin{align} \label{FF}
f(s) := 1 - s \cot{s},  \qquad \mu(s) := f'(s), \qquad \psi(s) := \frac{f(s)}{s^2}.
\end{align}

In this section, we collect some properties of $f$ and its derivatives such as $\psi$. These fine properties will be employed throughout the rest of this paper, the first and most important applications being the Propositions \ref{TH1} above and \ref{1PM} below.

It is well known that  we have for $-\pi < s < \pi$ (indeed for $s \in \C \setminus \{\pm k \pi; \ k \in \N^*\}$)
\begin{align} \label{IS0}
s \cot{s} = 1 - 2 s^2 \sum_{j = 1}^{+\infty} \frac{1}{(j \, \pi)^2 - s^2} = 1 - s^2  \sum_{j = 1}^{+\infty} (j \pi)^{-2} \left(  \frac{1}{1 -  \frac{s}{j \, \pi}} +  \frac{1}{1 +  \frac{s}{j \, \pi}} \right),
\end{align}
see for example \cite[\S 1.421 3, p.\,44]{GR15}  with slight modification. That is,
\begin{align} \label{IS}
-s \cot{s} = -1 + s^2  \sum_{j = 1}^{+\infty} (j \pi)^{-2} \left(  \frac{1}{1 -  \frac{s}{j \, \pi}} +  \frac{1}{1 +  \frac{s}{j \, \pi}}  \right) = -1 + \int_{-\frac{1}{\pi}}^{\frac{1}{\pi}} \frac{s^2}{1 - \lambda s} \, d\nu(\lambda),
\end{align}
where the positive finite measure on $[-\frac{1}{\pi}, \, \frac{1}{\pi}]$, $d\nu$ is given by
\begin{align*}
\nu = \sum_{j = 1}^{+\infty} \left( \frac{1}{j \, \pi} \right)^2 \left( \delta_{\frac{1}{j \, \pi}} + \delta_{-\frac{1}{j \, \pi}} \right) \quad \mbox{with $\delta_a$ the usual Dirac measure at point $a$.}
\end{align*}

The following lemma has been used by Gaveau (see \cite[Lemme 3, \, p.112]{G77}), which can be deduced directly by the first equality in \eqref{IS0}:

\begin{lem} \label{NL31}
The function $\mu$ defined in \eqref{FF} is an odd function, and a monotonely increasing diffeomorphism between $(-\pi, \, \pi)$ and $\R$. And we have
{\em \begin{align} \label{mu}
\mu(s) = \frac{2 s - \sin{2 s}}{2 \sin^2{s}}.
\end{align}}
\end{lem}

\begin{lem} \label{n32l}
We have
\begin{gather}
f^{''}(r) > \frac{f'(r)}{r} > 2 \psi(r) \geq 2 \psi(0) > 0, \quad 0 < r < \pi,    \label{Ii1} \\
\psi^{''}(r) > \frac{\psi'(r)}{r} \geq \lim_{r \longrightarrow 0} \frac{\psi'(r)}{r} > 0, \quad 0 < r < \pi,     \label{Ii2} \\
\psi(r) \, \psi^{''}(r) > 2 \psi'(r)^2, \quad 0 \leq r < \pi.    \label{Ii3}
\end{gather}
\end{lem}

\begin{proof}
A simple and direct proof for \eqref{Ii1} and \eqref{Ii2} is to use the fact that (see the first equality in \eqref{IS}, or \cite[\S 1.411 7, p.\,42]{GR15} with slight modification of notations)
\begin{align*}
f(s) = 1 - s \cot{s} = \sum_{j = 1}^{+\infty} b_j s^{2 j}, \quad \mbox{with } b_j > 0, \, \forall j.
\end{align*}

To prove \eqref{Ii3}, we apply \eqref{IS} and H\"older's inequality. Since
\begin{align*}
\psi(r) = \int_{-\frac{1}{\pi}}^{\frac{1}{\pi}} \frac{d\nu(\lambda)}{1 - r \lambda}, \quad \psi'(r) = \int_{-\frac{1}{\pi}}^{\frac{1}{\pi}} \frac{\lambda}{(1 - r \lambda)^2} \, d\nu(\lambda), \quad \psi^{''}(r) = 2 \int_{-\frac{1}{\pi}}^{\frac{1}{\pi}} \frac{\lambda^2}{(1 - r \lambda)^3} \, d\nu(\lambda),
\end{align*}
by H\"older's inequality, we obtain
\begin{align*}
2 \psi'(r)^2 < 2 \left( \int_{-\frac{1}{\pi}}^{\frac{1}{\pi}} \frac{|\lambda|}{(1 - r \lambda)^2} \, d\nu(\lambda) \right)^2 < \psi(r) \, \psi^{''}(r).
\end{align*}

This concludes the proof of the lemma.
\end{proof}

\begin{lem}
We have
\begin{align} \label{iN38}
\psi(r) > \sqrt{\frac{\psi'(r)}{r}}, \qquad \forall 0 \leq r < \pi.
\end{align}
\end{lem}

\begin{proof}
By the first equality in \eqref{IS0}, we have
\begin{align} \label{N32ei1}
\psi(r) = 2 \sum_{j = 1}^{+\infty} \frac{1}{(j \, \pi)^2 - r^2}, \qquad \frac{\psi'(r)}{r} = 4 \sum_{j = 1}^{+\infty} \left[(j \, \pi)^2 - r^2 \right]^{-2},
\end{align}
which implies immediately the lemma.
\end{proof}

Let $\V$ denote the unique solution of $\tan{s} = s$ in the interval $(\pi, \, 1.5 \, \pi)$, more precisely
\begin{align} \label{cN32}
\sqrt{2} \pi < 4.4933 < \V < 4.4935 < \frac{3}{2} \pi \mbox{ such that } \tan{\V} = \V.
\end{align}
We have the following simple lemma:

\begin{lem} \label{Ll34}
Let $\pi < r < \V$.  We have
\begin{gather}
0 >  \frac{1 -  r \cot{r}}{r^2} = \psi(r) = - 2 \left[ \frac{1}{r^2 - \pi^2} - \sum_{j = 2}^{+\infty} \frac{1}{(j \, \pi)^2 - r^2} \right], \label{32ei2} \\
\K_1 := \frac{\psi'(r)}{r} = 4 \sum_{j = 1}^{+\infty} \left[(j \, \pi)^2 - r^2 \right]^{-2} > 0, \label{32ei3}\\
\K_2 := r^{-1} \left( \frac{\psi'(r)}{r} \right)'  = -16  \left[ \frac{1}{[r^2 - \pi^2]^3} - \sum_{j = 2}^{+\infty} \frac{1}{[(j \, \pi)^2 - r^2]^3} \right] < 0.  \label{32ei4}
\end{gather}
Moreover, we have
\begin{align} \label{32ei5}
 2 \psi(r) \K_2 - 4 \K_1^2 < 0.
\end{align}
\end{lem}

\begin{proof}
We get easily \eqref{32ei2} and \eqref{32ei3} by \eqref{N32ei1}. Estimation \eqref{32ei4} follows directly from \eqref{32ei2} and the fact that $(j \, \pi)^2 - r^2 > r^2 - \pi^2$ for any $j \geq 2$. To get \eqref{32ei5}, by \eqref{32ei2} and \eqref{32ei4}, we have
\[
2 \psi(r) \, \K_2 < 2 \times 2 \times 16 \frac{1}{r^2 - \pi^2} \times \frac{1}{(r^2 - \pi^2)^3} = 4 \times \left( \frac{4}{(r^2 - \pi^2)^2} \right)^2 < 4 \, \K_1^2,
\]
where we have used \eqref{32ei3} in the last inequality. This completes the proof of Lemma \ref{Ll34}.
\end{proof}

Now we are in a position to provide a crucial ingredient of this paper: operator convexity of the function  $- s \cot{s}$. First we recall some notations, see for example \cite{U00}. Let $\T_1$, $\T_2$ be bounded self-adjoint operators on a Hilbert space $\mathcal{H}$. We denote $\T_1 \leq \T_2$ (resp. $\T_1 < \T_2$) if $(\T_1 v, v)  \leq (\T_2 v, v)$ (resp. $(\T_1 v, v)  < (\T_2 v, v)$) for all $v \neq 0$. Let $u$ be a real-valued continuous function defined on a finite interval $J$. Then $u$ is said to be operator convex on $J$ if
\begin{align*}
u\!\left(s \T_1 + (1 - s) \T_2 \right) \leq s  \, u\!\left(\T_1\right) + (1 - s) \, u\!\left(\T_2\right), \quad \forall 0 \leq s \leq 1,
\end{align*}
for all self-adjoint operators $\T_1$, $\T_2 \in \mathbb{B}(\mathcal{H})$ with spectra in $J$.

\begin{lem} \label{OCP}
The function $- s \cot{s}$ is operator convex on $(-\pi, \, \pi)$.
\end{lem}

\begin{proof}
This result is a direct consequence of Loewner's theorem (see for example \cite{L34}, \cite{BS55}, \cite{HJ91} or \cite{B97}) and  \eqref{IS}.
\end{proof}

\begin{remark}
It follows from \cite[Theorem 5.1]{U00}  that $- s \cot{s}$ is even strongly (so strictly) operator convex on $(-\pi, \, \pi)$ in the sense therein.
\end{remark}

{\bf Proof of Proposition \ref{TH1}:}

\medskip

Proposition \ref{TH1} is a straightforward consequence of  the operator convexity of $-s \cot{s}$ on $(-\pi, \, \pi)$.  \qed

\medskip

\renewcommand{\theequation}{\thesection.\arabic{equation}}
\section{Proof of Theorems \ref{MP1} and  \ref{TH2}}\label{SP2}
\setcounter{equation}{0}

\medskip

We will need the following lemma, which is a special case of \cite[Proposition III.5.3. (Ky Fan)]{B97}. We include an elementary proof for  the sake of clarity.

\begin{lem} \label{NL51}
Let $\mathrm{A}, \mathrm{B}$ be $q \times q$ skew-symmetric real matrices. If $\| i \mathrm{A} \| = \|\mathrm{A} \| = a$ and $\lambda = \alpha + i \beta$ ($\alpha, \beta \in \R$) is an eigenvalue of $i \mathrm{A} + \mathrm{B}$, then we have $|\alpha| \le a$.
\end{lem}

\begin{proof}
Suppose that the assertion claimed is false. Then there is an $|\alpha| > a$ and a vector $v \neq 0$ such that
\[
\left( i \mathrm{A} + \mathrm{B} \right) v = (\alpha + i \beta) v,
\]
which implies
\[
\left( \alpha  \, \mathbb{I}_q - i \mathrm{A} \right) v = \left( \mathrm{B} - i \beta  \, \mathbb{I}_q \right) v.
\]

Without loss of generality, we may assume that $\alpha > 0$. By the fact that $\alpha > \| i \mathrm{A} \|$, the matrix $\alpha  \, \mathbb{I}_q - i \mathrm{A}$ is positive definite and invertible, so $1$ is an eigenvalue of
\[
\left( \alpha \, \mathbb{I}_q - i \mathrm{A} \right)^{-1} \left( \mathrm{B} - i \beta  \, \mathbb{I}_q \right).
\]

Set $\mathrm{C} = \left( \alpha  \, \mathbb{I}_q - i \mathrm{A} \right)^{-1}$ which is positive definite. Recall that for any $q \times q$ matrices $A$ and $B$, $AB$ and $BA$ have the same eigenvalues, cf. for example \cite[p.\,11]{B97}. Hence $1$ is an eigenvalue of
\[
\mathrm{C}^{\frac{1}{2}} \left( \mathrm{B} - i \beta  \, \mathbb{I}_q \right) \mathrm{C}^{\frac{1}{2}}
\]
which is obviously skew-Hermitian. This leads to a contradiction.
\end{proof}

In order to use the method of stationary phase,  we need the following key ingredient:

\begin{prop}   \label{1PM}
Let $\tau_0 \in \OA$, $v_0 \in \mathbb{S}^{m - 1}$, and $x_0 \in \R^{q}$. Then the even function
\begin{eqnarray*}
\Xi(\tau_0, v_0; x_0; s) := - \Re \langle U(\tau_0 + i (s v_0)) \cot{U(\tau_0 + i (s v_0))} x_0, x_0 \rangle, \quad s \in \R,
\end{eqnarray*}
is nonincreasing w.r.t. $|s|$. Furthermore, under the assumption of Theorem \ref{TH2}, $\Xi(\theta_0, v_0; x_0; s)$ is strictly decreasing w.r.t. $|s|$.
\end{prop}

\begin{proof}
The first statement is trivial if $x_0 = 0$; assume therefore $x_0 \neq 0$.

By \eqref{IS0}, we have
\begin{align} \label{IS1}
- r \cot{r} = -1 + \sum_{j = 1}^{+\infty} \left(   \frac{1}{1 -  \frac{r}{j \, \pi}} +  \frac{1}{1 +  \frac{r}{j \, \pi}}  - 2 \right).
\end{align}

Now, let us define
\begin{align} \label{HDG}
h_{\gamma}(z) := \frac{1}{1 - \gamma \, z}, \qquad -\frac{1}{\pi} \le \gamma \le \frac{1}{\pi}, \  -\pi  <  \Re z  < \pi,
\end{align}
and
\begin{align*}
\overrightarrow{h_{\gamma}}(\tau_0, s v_0) := h_{\gamma}(\widetilde{U}(i \tau_0 - s v_0)) = h_{\gamma}(U(\tau_0 + i s v_0)).
\end{align*}

Remark that
\[
2 \Re \overrightarrow{h_{\gamma}}(\tau_0, s v_0) = h_{\gamma}(U(\tau_0 + i s v_0)) + h^*_{\gamma}(U(\tau_0 + i s v_0)) = h_{\gamma}(U(\tau_0 + i s v_0)) + h_{\gamma}(U(\tau_0 - i s v_0))
\]
which implies that
\[
2 \, \left( \Xi(\tau_0, v_0; x_0; s) - \Xi(\tau_0, v_0; x_0; 0) \right) = \sum_{j = 1}^{+\infty} \left[ 2 \widetilde{h}(\frac{1}{j \, \pi}; \tau_0, v_0; x_0; s) + 2 \widetilde{h}(-\frac{1}{j \, \pi}; \tau_0, v_0; x_0; s) \right]
\]
where we have defined for $0 \neq \gamma \in [-\frac{1}{\pi}, \, \frac{1}{\pi}]$, $\tau_0 \in \OA$, $v_0 \in \mathbb{S}^{m - 1}$ and $x_0 \in \R^{q}$, the even function
\begin{align} \label{nnn1}
2 \widetilde{h}(\gamma; \tau_0, v_0; x_0; s) := \left\langle \left( \overrightarrow{h_{\gamma}}(\tau_0, s v_0) + \overrightarrow{h_{\gamma}}(\tau_0, - s v_0)  - 2 \overrightarrow{h_{\gamma}}(\tau_0, 0)  \right) x_0, x_0 \right\rangle, \quad s \in \R.
\end{align}

It remains to show that $2 \widetilde{h}(\gamma; \tau_0, v_0; x_0; s)$ is decreasing w.r.t. $|s|$.  Notice that $1 - \gamma \, U(\tau_0)$ is positive definite, a simple calculation gets
\begin{align} \label{nnn2}
&\overrightarrow{h_{\gamma}}(\tau_0, s v_0) + \overrightarrow{h_{\gamma}}(\tau_0, - s v_0)  - 2 \overrightarrow{h_{\gamma}}(\tau_0, 0) = (1 - \gamma \, U(\tau_0))^{-\frac{1}{2}} \, \overrightarrow{W_{\gamma}} (s)  \, (1 - \gamma \, U(\tau_0))^{-\frac{1}{2}}
\end{align}
with
\begin{align*}
\overrightarrow{W_{\gamma}} (s) &:= h_{\gamma}\left(-s \, (1 - \gamma \, U(\tau_0) )^{-\frac{1}{2}}  \,  \widetilde{U}(v_0) \,  (1 - \gamma \, U(\tau_0))^{-\frac{1}{2}} \right)  \\
&+  h_{\gamma}\left(s \, (1 - \gamma \, U(\tau_0))^{-\frac{1}{2}}  \,  \widetilde{U}(v_0) \,  (1 - \gamma \, U(\tau_0))^{-\frac{1}{2}} \right) - 2 \, \mathbb{I}_q.
\end{align*}

Since $(1 - \gamma \, U(\tau_0))^{-\frac{1}{2}}  \,  \widetilde{U}(v_0) \,  (1 - \gamma \, U(\tau_0))^{-\frac{1}{2}}$ is skew-Hermitian, there exists a unitary matrix $P_{\gamma} = P_{\gamma}(\tau_0, v_0)$ such that
\begin{align*}
P_{\gamma}^*  \, (1 - \gamma \, U(\tau_0))^{-\frac{1}{2}}  \,  \widetilde{U}(v_0) \,  (1 - \gamma \, U(\tau_0))^{-\frac{1}{2}} \,  P_{\gamma} = \left(
	\begin{array}{ccc}
	i \varrho_1& \ & \ \\
	\  & \ddots & \ \\
	\ & \ & i \varrho_q\\
	\end{array}
	\right)
\end{align*}
where $\varrho_j = \varrho_j(\gamma, \tau_0, v_0) \in \R$. Thus
\begin{align} \label{nnn3}
\overrightarrow{W_{\gamma}} (s) = - 2 P_{\gamma}  \left(
	\begin{array}{ccc}
	1 - \frac{1}{1 + (\gamma \, \varrho_1)^2 \, s^2} & \ & \ \\
	\  & \ddots & \ \\
	\ & \ & 1 - \frac{1}{1 + (\gamma \, \varrho_q )^2 \, s^2}\\
	\end{array}
	\right)  P_{\gamma}^*,
\end{align}
which implies that $2 \, \widetilde{h}(\gamma; \tau_0, v_0; x_0; s)$ decreases w.r.t. $|s|$, so does $\Xi$.

To obtain the conclusion under the assumption of Theorem \ref{TH2}, we argue in a similar way. In fact, in such case, by replacing $s$ by $i s$ with $|s| \ll 1$, it is not hard to find that there exists at least a $\gamma \in \{ \pm \frac{1}{j \pi}; j = 1, 2, \ldots,\}$ such that (see also \eqref{nEs2} below)
\[
2 \, \widetilde{h}(\gamma; \theta_0, v_0; x_0; i s) > 0, \quad \forall 0 < |s| \ll 1,
\]
which implies easily, via \eqref{nnn1}-\eqref{nnn3}, that $\Xi(\theta_0, v_0; x_0; s)$ is strictly decreasing w.r.t. $|s|$.
\end{proof}

We are in a position to provide {\bf the proof of Theorems \ref{MP1} and  \ref{TH2}}:

\medskip

Fix $o \neq g \in \G$. Let us define
\begin{align*}
\D := i \, \OA + \R^m = \{ i \vartheta + \zeta; \ \vartheta \in \OA, \ \zeta \in \R^m \} \subset \C^m.
\end{align*}
It is clear that $V(\lambda) \exp\left\{-\widetilde{\phi}(g; \lambda)/4\right\}$ is holomorphic on $\D$.
For $a > 0$ small enough, set
\[
\D_a := \{ i \vartheta + \zeta; \ \zeta \in \R^m, \  \vartheta \in \OA, \ \mathrm{dist}(\vartheta, \partial \OA) \geq a \} \subset \D.
\]
By the first conclusion of Proposition \ref{1PM},  $\exp\left\{-\widetilde{\phi}(g; \lambda)/4\right\}$ is bounded in $\D_a$ for any fixed $g \in \G$.

On the other hand, we have the following lemma.

\begin{lem} \label{NL53}
$V(\lambda)$ is bounded and decays exponentially at infinity in $\D_a$.
\end{lem}

\begin{proof}
By the definition of $V$ (see \eqref{2c'}), it suffices to show that
\begin{align*}
\mathrm{det} \frac{U(\lambda)}{\sinh{U(\lambda)}} = \mathrm{det} \frac{\widetilde{U}(\lambda)}{\sin{\widetilde{U}(\lambda)}}
\end{align*}
is bounded and decays exponentially at infinity in $\D_a$.

Recall that if $\lambda_1, \ldots, \lambda_q$ (not necessarily distinct) are eigenvalues of $q \times q$ matrix $A$ with $\lambda_j \not\in \{\pm \pi, \pm 2 \pi, \ldots, \}$, we have
\[
\mathrm{det} \frac{A}{\sin{A}} = \prod_{j = 1}^q \frac{\lambda_j}{\sin{\lambda_j}}.
\]
Obviously, there exists a constant $0 < c(a) < \pi$ such that $\| U(\vartheta) \| = \| \widetilde{U}(i \vartheta) \| \leq c(a)$ for any $i \vartheta \in \D_a$.  By the fact that $\|  \widetilde{U}(i \vartheta + \zeta) \| \geq  \| \widetilde{U}(\zeta) \| - \| \widetilde{U}(i \vartheta) \|$, the desired result is a direct consequence of Lemma \ref{NL51}.
\end{proof}

{\em Added:} We can get the following much more stronger estimate without using Lemma \ref{NL51}:
\[
|V(i \theta + \lambda)| \leq V(i \theta) \, V(\lambda), \qquad \forall \theta \in \OA, \  \lambda \in \R^m.
\]

\medskip

Thus, we can deform the contour and write
\begin{align} \label{ME52}
p(\delta_{h^{-1/2}}(g)) = \int_{\R^m} V(\lambda) e^{-\frac{\widetilde{\phi}(g; \lambda)}{4 h}} \, d\lambda = \int_{\R^m} V(i \vartheta + \lambda) e^{-\frac{\widetilde{\phi}(g; i \vartheta + \lambda)}{4 h}} \, d\lambda,
\end{align}
for any $\vartheta \in \OA$.

\medskip

\subsection{End of the proof of Theorem \ref{MP1}}

\medskip

Using the first claim of Proposition \ref{1PM}, the RHS of \eqref{ME52} is controlled by
\begin{align*}
\int_{\R^m} \left| V(i \vartheta + \lambda) e^{-\frac{\widetilde{\phi}(g; i \vartheta + \lambda)}{4 h}} \right| \, d\lambda \leq \int_{\R^m} \left| V(i \vartheta + \lambda) \right| e^{-\frac{\widetilde{\phi}(g; i \vartheta)}{4 h}} \, d\lambda.
\end{align*}
It follows from Lemma \ref{NL53} and the definition of $\phi$ (see \eqref{2cn}) that the last term is majored by $C(\vartheta) e^{-\frac{\phi(g; \vartheta)}{4 h}}$.

Recall that $p_h(g) = C h^{-\frac{q}{2} - m} p(\delta_{h^{-1/2}}(g))$ (see \eqref{sp} and \eqref{2c0}), it deduces from Varadhan's formula (see \eqref{VF}) that
\begin{align*}
-d(g)^2 \leq - \phi(g; \vartheta), \qquad \forall \vartheta \in \OA.
\end{align*}

This finishes the proof of Theorem \ref{MP1}.    \qed

\medskip

\subsection{End of the proof of Theorem \ref{TH2}}

\medskip

Under our assumptions, take $\vartheta = \theta_0$ in \eqref{ME52} which is uniquely determined by \eqref{m}. By Lemma \ref{NL53} and the second claim in Proposition \ref{1PM}, we can use the method of stationary phase (see for example \cite[Theorem 7.7.5]{H90}) to study the asymptotic behavior of $p_h(g_0) = C h^{-\frac{q}{2} - m} p(\delta_{h^{-1/2}}(g_0))$ as $h \longrightarrow 0^+$. More precisely, we have
\begin{align*}
p_h(g_0) = C \,  (8 \pi)^{\frac{m}{2}} \, h^{-\frac{q + m}{2}} V(i \theta_0) e^{-\frac{\phi(g_0; \theta_0)}{4 h}} \left[  \mathrm{det} \left(- \mathrm{Hess}  \, \phi(g_0; \theta_0) \right) \right]^{-\frac{1}{2}} \left( 1 + O(h) \right), \qquad h \longrightarrow 0^+.
\end{align*}

Using Varadhan's formula again, we get $d(g_0)^2 = \phi(g_0; \theta_0)$, which finishes the proof of Theorem \ref{TH2}. \qed

\medskip

\renewcommand{\theequation}{\thesection.\arabic{equation}}
\section{Proof of Theorem \ref{TH3N}} \label{S5s}
\setcounter{equation}{0}

\medskip

In order to show Theorem \ref{TH3N}, we need two simple lemmas. The first one is the following:

\begin{lem} \label{DNL1}
Let $1 \leq k \leq m$. For $\vartheta \in \mathcal{V}$ (defined by \eqref{nD1V}) and $x \in \R^q$, we have
{\begin{align} \label{nbe1}
& \frac{\partial}{\partial \vartheta_k} \left\langle - U(\vartheta) \, \cot{U(\vartheta)} \, x, \, x \right\rangle \nonumber \\
&= \sum_{j = 1}^{+\infty} \frac{1}{j \pi} \left\{ \left\langle \left( 1 - \frac{U(\vartheta)}{j \pi} \right)^{-1} (i \, U^{(k)})  \left( 1 - \frac{U(\vartheta)}{j \pi} \right)^{-1} x, \, x \right\rangle  \right. \nonumber \\
&\qquad \mbox{} \qquad \mbox{} \quad \left. -  \left\langle \left( 1 + \frac{U(\vartheta)}{j \pi} \right)^{-1} (i \, U^{(k)})  \left( 1 + \frac{U(\vartheta)}{j \pi} \right)^{-1} x, \, x \right\rangle  \right\}.
\end{align}}
\end{lem}

\begin{proof}
It suffices to use \eqref{IS1} and the fact that
\[
\frac{\partial}{\partial \vartheta_k} (1 - \gamma U(\vartheta))^{-1} = \gamma \, (1 - \gamma U(\vartheta))^{-1} (i U^{(k)}) (1 - \gamma U(\vartheta))^{-1}.
\]
\end{proof}

And the second one is the following:

\begin{lem}
Let $g = (x, t)$ and $\theta_0$ be a critical point of $\phi(g; \cdot)$ in $\OA$. For $|s| \geq 1$ and $v \in \R^q$, set
\[
\xNx(s) := h_{\frac{1}{s \pi}}(U(\theta_0)) \, x = \left( 1 - \frac{U(\theta_0)}{s \pi} \right)^{-1} x, \qquad \xNx_*(s) := U(v) \, \xNx(s).
\]
Then
we have
\begin{align} \label{nEs1}
4 t \cdot v = \sum_{j = 1}^{+\infty} \frac{1}{j \pi} \Big\{  \left\langle  U(v) \, \xNx(j), \, \xNx(j)  \right\rangle_{\C}  -  \left\langle  U(v) \,  \xNx(-j), \, \xNx(-j)  \right\rangle_{\C} \Big\}
\end{align}
and
\begin{align} \label{nEs2}
&\langle \He \phi(g; \theta_0) v, \, v \rangle \nonumber \\
&= \sum_{j = 1}^{+\infty} \frac{2}{(j \pi)^2}  \left\{  \left\langle  h_{\frac{1}{j \pi}}(U(\theta_0)) \, \xNx_*(j), \, \xNx_*(j) \right\rangle_{\C} +  \left\langle  h_{-\frac{1}{j \pi}}(U(\theta_0)) \, \xNx_*(-j), \, \xNx_*(-j)  \right\rangle_{\C} \right\}.
\end{align}
In particular, we have $t = 0$ if $\theta_0 = 0$.
\end{lem}

\begin{proof}
The proof is easy. Indeed, \eqref{nEs1} is a direct consequence of \eqref{nbe1} and the fact that
\[
4 t = - \nabla_{\theta = \theta_0} \langle U(\theta) \cot{U(\theta)} x, \, x \rangle.
\]

To prove \eqref{nEs2}, by \eqref{IS1} and \eqref{HDG}, it suffices to use
\begin{align*}
\frac{d^2}{ds^2}\Big|_{s = 0} h_{\gamma}(U(\theta + s \, v))  = 2 \gamma^2 \, h_{\gamma}(U(\theta)) U(v) h_{\gamma}(U(\theta)) U(v) h_{\gamma}(U(\theta)).
\end{align*}
\end{proof}

The following proposition is an improvement of Theorem \ref{TH3N}:

\begin{prop} \label{nP65}
Let $g = (x, t)$, $v_0 \in \s^{m - 1}$ and $\theta_0$ be a critical point of $\phi(g; \cdot)$ in $\OA$. Then the following are equivalent: \\
(a) $\langle \He \phi(g; \theta_0) v_0, \, v_0 \rangle = 0$; \\
(b) $t \cdot v_0 = 0$ and $U(v_0) U(\theta_0)^j x = 0$ for all $j = 0, 1, \ldots$; \\
(c) $t \cdot v_0 = 0$ and $\phi(g; \theta_0 + s v_0) = \phi(g; \theta_0)$ for all $s \in \R$ satisfying $\theta_0 + s v_0 \in \OA$; \\
(d) $\phi(g; \theta_0 + s_0 v_0) = \phi(g; \theta_0)$ for some $s_0 \neq 0$ such that $\theta_0 + s_0 v_0 \in \OA$.
\end{prop}

\begin{proof}
We first prove the equivalence of (a) and (b). Notice that $h_{\gamma}(U(\theta_0))$ is positive definite for all $-\pi^{-1} \leq \gamma \leq \pi^{-1}$. By \eqref{nEs2}, the assertion (a) is equivalent to the fact that
\[
0 = U(v_0) \left(1 - \frac{U(\theta_0)}{k \pi} \right)^{-1} x = \sum_{j = 0}^{+\infty} (k \pi)^{-j} \, U(v_0) \, U(\theta_0)^j \, x, \qquad \forall k = \pm 1, \pm 2, \ldots,
\]
which is equivalent to, by induction and a limiting argument, $U(v_0) U(\theta_0)^j x = 0$ for all $j = 0, 1, \ldots$. Moreover, it follows from \eqref{nEs1} that $t \cdot v_0 = 0$.

Now we show that (b) implies (c). Notice that the function of $z$, $\phi(g; \theta_0 + z v_0)$ is holomorphic in a neighborhood of $\{s \in \R; \  \theta_0 + s v_0 \in \OA\}$. It remains to show that $\phi(g; \theta_0 + s v_0) = \phi(g; \theta_0)$ for $|s| \ll 1$. Indeed, by \eqref{IS1}, it suffices to remark that it deduces from the second condition in (b) that we have for any $-\pi^{-1} \leq \gamma \leq \pi^{-1}$ and $\| s \gamma U(v_0) (1 - \gamma \, U(\theta_0))^{-1} \| < 1$
\begin{align*}
&\left[ (1 - \gamma (U(\theta_0 + s v_0)))^{-1} - (1 - \gamma \, U(\theta_0))^{-1} \right] x \\
&= \sum_{j = 1}^{+\infty} (s \gamma)^j (1 - \gamma \, U(\theta_0))^{-1} \left[ U(v_0) (1 - \gamma \, U(\theta_0))^{-1} \right]^j x = 0,
\end{align*}
where we have used in the last equality
\[
U(v_0) (1 - \gamma \, U(\theta_0))^{-1} \, x = \sum_{j = 0}^{+\infty} U(v_0) \, \left( \gamma \, U(\theta_0) \right)^j \, x = 0.
\]

It is trivial that (c) implies (d). To show that (d) implies (a), it is enough to remark that $\phi(g; \theta_0 + s v_0) = \phi(g; \theta_0)$ for any $s$ between $0$ and $s_0$ by the concavity of $\phi(g; \cdot)$.
\end{proof}

\begin{remark}
The second condition in (b) is inspired by \cite[Lemma 10]{MM16}.
\end{remark}

\medskip

For $g \in \OM_2$ defined by \eqref{OM2}, set
\begin{align} \label{Som2}
\Sigma_g  := \left\{  \theta; \,  \theta \mbox{ is a critical point (i.e. global maximizer) of  $\phi(g; \cdot)$ in $\OA$}  \right\}.
\end{align}
In particular, we have the following simple observation that should be useful to study short-time asymptotic behavior of the heat kernel:

\begin{cor}
Let $g = (x, t) \in \OM_2$. There exists $\theta(g) \in \Sigma_g$ and $k(g)$-dimensional linear subspace of $\R^m$, $\Pi_g$ such that
\[
\Sigma_g = \OA \cap \left( \theta(g) + \Pi_g \right), \qquad t \cdot v = 0, \  \forall v \in  \Pi_g.
\]
Moreover,  the orthogonal projection of $x$ on $\pi^2$-eigenspace of $U^2( \theta(g) + v)$ is zero for any $v  \in  \Pi_g$ satisfying $\theta(g) + v \in \partial \OA$. Furthermore, if $0 \in \Sigma_g$, then $t = 0$.
\end{cor}

Now, we provide a sufficient condition so that $g \in \M$ which should be useful in application. We define as in \cite[\S 3]{MM16} $\mathcal{E} := \{ \widetilde{U}(\theta); \, \theta \in \R^m \}$ and
\[
\mathcal{Q} := \{ \theta; \, \widetilde{U}(\theta) \mbox{ has maximal number of distinct eigenvalues among the elements of $\mathcal{E}$} \}
\]
which is a homogeneous Zariski-open subset of  $\R^m$.

\begin{prop} \label{nPP510}
Let $\theta \in \OA \cap \mathcal{Q}$ and let $x$ be such that the orthogonal projection of $x$ on each eigenspace of $U^2(\theta)$ is nonzero. Then for $t =  -\frac{1}{4} \nabla_{\theta} \langle U(\theta) \cot{U(\theta)} \, x, x  \rangle$, we have $(x, t) \in \M$.
\end{prop}

\begin{proof}
This result is a direct consequence of \cite[Lemma 10]{MM16} and Proposition \ref{nP65}.
\end{proof}

To finish this section, we provide the

\subsection{Proof of Proposition \ref{nPr210}}

\begin{proof}
Obviously, it suffices to show the first claim. Let $t_0 \in \R^m$ and $x_0 \in \R^q \setminus\! \{0\}$. Suppose that for any $\tau \in \partial \OA$, the orthogonal projection of $x_0$ on $\pi^2$-eigenspace of $U^2(\tau)$ is nonzero, we shall prove $g_0 = (x_0, t_0) \in \M$.

By Remark \ref{nRk21}, the reference function $\phi(g_0; \cdot)$ attains its maximum in $\OA$, since $\mathrm{cl} \phi(g_0; \tau) = -\infty$ for any $\tau \in \partial \OA$. Then $g_0 \in \OM$.

Let $\theta_0 \in \OA$ be a critical point of $\phi(g_0; \cdot)$. It remains to show that $\theta_0$ is nondegenerate. By contradiction, assume that $\theta_0$ is degenerate. Then, by (c) in Proposition \ref{nP65}, there exist $v_0 \in \s^{m - 1}$ and $s_0 > 0$ such that $\theta_0 + s_0 \, v_0 \in \partial \OA$,
\[
\theta_0 + s \, v_0 \in \OA, \quad \phi(g_0; \theta_0) = \phi(g_0; \theta_0 + s \, v_0), \quad \forall 0 \le s < s_0.
\]
Using Remark \ref{nRk21} again, we get that $\mathrm{cl}\phi(g_0; \theta_0 + s_0 \, v_0) \neq - \infty$, which gives a contradiction.

This finishes the proof of Proposition  \ref{nPr210}.
\end{proof}

\medskip

\renewcommand{\theequation}{\thesection.\arabic{equation}}
\section{Proof of Theorem \ref{THN4}}  \label{PTHN4}
\setcounter{equation}{0}

\medskip

Let us begin by the proof of the characterization of ``good'' normal geodesics. Set
\[
\xNx(j) := \left( 1 - \frac{U(\theta_0)}{j \pi} \right)^{-1} x_0, \qquad j \in \Z \setminus \{0\}.
\]
By \eqref{GEn2} and \eqref{nbe1}, it suffices to show that for all $1 \leq k \leq m$, we have
\begin{align} \label{nXE1}
&\sum_{j = 1}^{+\infty} \frac{1}{j \pi} \Big\{ \left\langle U^{(k)}  \, \xNx(j) , \, \xNx(j) \right\rangle_{\C}  -  \left\langle  U^{(k)} \, \xNx(-j), \, \xNx(-j) \right\rangle_{\C}  \Big\} \nonumber \\
&= \frac{2}{i} \int_0^1 \left\langle U^{(k)} \frac{\sin{(s U(\theta_0))}}{\sin{U(\theta_0)}} e^{(s - 1) \widetilde{U}(\theta_0)} \, x_0,  \,  \frac{U(\theta_0)}{\sin{U(\theta_0)}} e^{(2 s - 1) \widetilde{U}(\theta_0)} \, x_0 \right\rangle_{\C} \, ds.
\end{align}

We show the equality via the Spectral Theorem. More precisely, let $\lambda_1 < \ldots < \lambda_{k_0}$ denote the distinct eigenvalues of the Hermitian matrix $U(\theta_0)$, and $u_l \in \C^q$ ($1 \leq l \leq k_0$) the orthogonal projection of $x_0$ onto the $\lambda_l$-eigenspace of $U(\theta_0)$.
Then $x_0 = \sum_{l = 1}^{k_0} u_l$, and we have for a suitable function $W$,
\[
W(U(\theta_0))x_0 = \sum_{l = 1}^{k_0} W(\lambda_l) u_l \mbox{ \ and in particular \ } \xNx(j) = \sum_{l = 1}^{k_0} (1 - \frac{\lambda_l}{j \pi})^{-1} u_l.
\]

Now, the LHS of \eqref{nXE1} equals to
\begin{align} \label{nXE2}
\sum_{\alpha, \beta = 1}^{k_0} C(\alpha, \beta) \, \langle U^{(k)} u_{\alpha}, \, u_{\beta} \rangle_{\C}
\end{align}
with
\[
C(\alpha, \beta) := \sum_{j = 1}^{+\infty} \frac{1}{j \pi} \left[ (1 - \frac{\lambda_{\alpha}}{j \pi})^{-1} (1 - \frac{\lambda_{\beta}}{j \pi})^{-1} - (1 + \frac{\lambda_{\alpha}}{j \pi})^{-1} (1 + \frac{\lambda_{\beta}}{j \pi})^{-1} \right].
\]

In the case where $\alpha \neq \beta$, we have
\begin{align*}
C(\alpha, \beta) &= \frac{2}{\lambda_{\alpha} - \lambda_{\beta}} \sum_{j = 1}^{+\infty} \frac{(\frac{\lambda_{\alpha}}{j \pi})^2 - (\frac{\lambda_{\beta}}{j \pi})^2}{[ 1 - (\frac{\lambda_{\alpha}}{j \pi})^2] \cdot [ 1 - (\frac{\lambda_{\beta}}{j \pi})^2]} \\
&= \frac{2}{\lambda_{\alpha} - \lambda_{\beta}} \sum_{j = 1}^{+\infty} \left[ \frac{\lambda^2_{\alpha}}{(j \pi)^2 - \lambda^2_{\alpha}} - \frac{\lambda^2_{\beta}}{(j \pi)^2 - \lambda^2_{\beta}} \right] \\
&= \frac{\lambda_{\beta} \cot{\lambda_{\beta}} - \lambda_{\alpha} \cot{\lambda_{\alpha}} }{\lambda_{\alpha} - \lambda_{\beta}}
\end{align*}
where we have used \eqref{IS0} in the last equality.

In the opposite case $\alpha = \beta$, by the fact that $f'(s) = \mu(s)$ (cf. \eqref{FF}), we have
\begin{align*}
C(\alpha, \alpha) = 4 \lambda_{\alpha} \sum_{j = 1}^{+\infty} \frac{(j \pi)^2}{((j \pi)^2 - \lambda^2_{\alpha})^2} = f'(\lambda_{\alpha}) = \mu(\lambda_{\alpha})
\end{align*}
by using \eqref{IS0} again.

In conclusion, the LHS of \eqref{nXE1} can be written as
\begin{align} \label{EFT}
\sum_{\alpha = 1}^{k_0} \mu(\lambda_{\alpha}) \, \langle U^{(k)} u_{\alpha}, \, u_{\alpha} \rangle_{\C} + \sum_{\alpha \neq \beta} \frac{\lambda_{\beta} \cot{\lambda_{\beta}} - \lambda_{\alpha} \cot{\lambda_{\alpha}} }{\lambda_{\alpha} - \lambda_{\beta}} \, \langle U^{(k)} u_{\alpha}, \, u_{\beta} \rangle_{\C}.
\end{align}

On the other hand, the RHS of \eqref{nXE1} can be written as
\begin{align*}
\sum_{\alpha, \beta = 1}^{k_0} D(\alpha, \beta) \, \langle U^{(k)} u_{\alpha}, \, u_{\beta} \rangle_{\C}
\end{align*}
with
\begin{align*}
D(\alpha, \beta) &:= \frac{2}{i} \frac{\lambda_{\beta}}{\sin{\lambda_{\beta}}} e^{i (\lambda_{\alpha} - \lambda_{\beta})} \int_0^1 \frac{\sin{(s \lambda_{\alpha})}}{\sin{\lambda_{\alpha}}} e^{i s (2 \lambda_{\beta} - \lambda_{\alpha})} \, ds  \\
&= - \frac{\lambda_{\beta}}{\sin{\lambda_{\beta}}} e^{i (\lambda_{\alpha} - \lambda_{\beta})} \int_0^1 \frac{e^{2 i s \lambda_{\beta}} - e^{2 i s (\lambda_{\beta} - \lambda_{\alpha})}}{\sin{\lambda_{\alpha}}} \, ds.
\end{align*}

A direct calculation implies that
\[
D(\alpha, \beta) = C(\alpha, \beta) - i, \qquad 1 \leq \alpha, \beta \leq k_0.
\]
Thus, the difference between the RHS and LHS of \eqref{nXE1} is
\[
- i \sum_{\alpha, \beta = 1}^{k_0} \langle U^{(k)} u_{\alpha}, \, u_{\beta} \rangle_{\C} = -i \, \langle U^{(k)} x_0, \, x_0 \rangle = 0,
\]
because the real matrix $U^{(k)}$ is skew-symmetric.

This completes the proof of the characterization.    \qed

\medskip

It remains to provide the

\subsection{Proof of \eqref{UED}}

By Subsection \ref{nSSn23} (II)-(III), \eqref{GEn3} implies that $d(g_0)^2 \le \ell(\gamma)^2 = \left| \frac{U(\theta_0)}{\sin{U(\theta_0)}} x_0 \right|^2$.
It remains to show the following:

\begin{prop} \label{NC51}
Let $g = (x, t)$, and $\vartheta \in  \mathcal{V}$ be a critical point of $\phi(g; \cdot)$. Then
\begin{align} \label{nbe2}
\phi(g; \vartheta) = \left| \frac{U(\vartheta)}{\sin{U(\vartheta)}} x \right|^2 = \left\langle \left( \frac{U(\vartheta)}{\sin{U(\vartheta)}} \right)^2 x, \, x \right\rangle.
\end{align}
\end{prop}

\begin{proof}
By the definition of $\vartheta$, we have $4 \, t = - \nabla_{\vartheta} \langle U(\vartheta) \, \cot{U(\vartheta)} x, \, x \rangle$. Thus
\[
\phi(g; \vartheta) = \langle U(\vartheta) \, \cot{U(\vartheta)} x, \, x \rangle - \vartheta \cdot \nabla_{\vartheta} \langle U(\vartheta) \, \cot{U(\vartheta)} x, \, x \rangle.
\]

It follows from the definition of $U(\vartheta)$ (see \eqref{2c'}) and \eqref{nbe1} that
\[
-\vartheta \cdot \nabla_{\vartheta} \langle U(\vartheta) \, \cot{U(\vartheta)} x, \, x \rangle = \sum_{j = 1}^{+\infty} \frac{1}{j \pi} \left\langle \left( \frac{U(\vartheta)}{  \left( 1 - \frac{U(\vartheta)}{j \pi} \right)^2} -  \frac{U(\vartheta)}{\left( 1 + \frac{U(\vartheta)}{j \pi} \right)^2} \right) \, x, \, x \right\rangle.
\]

Moreover, we have (cf. \cite[\S 1.4.22 4, p. 44]{GR15})
\begin{align}
\left( \frac{s}{\sin{s}} \right)^2 = 1 + s^2 \sum_{j = 1}^{+\infty} \left[ (j \pi - s)^{-2} + (j \pi + s )^{-2} \right].
\end{align}

By the Spectral Theorem and the first equality in \eqref{IS0}, it suffices to show that we have for any $j = 1, 2, \ldots$,
\begin{align*}
-2 \frac{s^2}{(j \pi)^2 - s^2} + \frac{s}{j \pi} \left[ (1 - \frac{s}{j \pi})^{-2} -  (1 + \frac{s}{j \pi})^{-2} \right]  = s^2 \left[ (j \pi - s)^{-2} + (j \pi + s )^{-2} \right],
\end{align*}
which is easy to check.

This finishes the proof of the proposition.
\end{proof}

\medskip

\renewcommand{\theequation}{\thesection.\arabic{equation}}
\section{Proof of Theorems \ref{THN6} and \ref{aT121}}  \label{Pthn6}
\setcounter{equation}{0}

\medskip

In order to show Theorem \ref{THN6}, we need the following:

\begin{prop} \label{NPA}
Let $\Omega \subset \R^m$ be an open, bounded, convex set. Assume that the function
\[
h: \, \R^k \times \Omega \longrightarrow \R
\]
is continuous and for any $\xx \in \R^k$, $h(\xx, \cdot)$ is concave in $\Omega$. Then the function defined by
\begin{align*}
H(\xx) := \sup_{\tau \in \Omega} h(\xx; \tau)
\end{align*}
is continuous on $\R^k$.
\end{prop}

\begin{proof}
We argue by contradiction. Suppose that $H$ is discontinuous at $\xx_0$. Without loss of generality, we may assume that $H(\xx_0) = 0$.

Then, there exists $\epsilon_0 > 0$ and $\{ \xx_j \}_{j = 1}^{+\infty}$ such that
\begin{align} \label{nAH}
\xx_j \longrightarrow \xx_0, \qquad |H(\xx_j)| \geq 4 \epsilon_0, \  \forall j.
\end{align}

First we claim that $H(\xx_j) \geq 4 \epsilon_0$ if $j$ is large enough. Indeed, by the definition of $H$, there exists $\tau_0 \in \Omega$ such that
\[
-\frac{\epsilon_0}{4} \leq h(\xx_0, \tau_0) < 0.
\]
Now the continuity of $h$ implies that for some neighborhood of $(\xx_0, \tau_0)$, $V_{\xx_0} \times V_{\tau_0}$, we have
\[
|h(\xx, \tau)| \leq \epsilon_0, \qquad \forall (\xx, \tau) \in V_{\xx_0} \times V_{\tau_0}.
\]
Therefore, for $j \gg 1$ we have $\xx_j \in V_{\xx_0}$ and
\[
H(\xx_j) = \sup_{\tau \in \Omega} h(\xx_j, \tau) \geq \sup_{\tau \in V_{\tau_0}} h(\xx_j, \tau) \geq - \epsilon_0.
\]
Combining this with \eqref{nAH}, we obtain the assertion.

From now on, we may assume that $\xx_j \in V_{\xx_0}$ and $H(\xx_j) \geq 4 \epsilon_0$ for all $j \geq 1$. By the definition of $H$, there exists $\tau_j \in \Omega$ such that $h(\xx_j, \tau_j) \geq 3 \epsilon_0$. By the compactness of $\overline{\Omega}$, up to subsequences, we may further assume that $\tau_j \longrightarrow \tau_* \in \overline{\Omega}$. It follows from the concavity of $h(\xx_j, \cdot)$ that
\[
h(\xx_j, \frac{\tau_0 + \tau_j}{2}) \geq \frac{1}{2} \left( h(\xx_j, \tau_0) + h(\xx_j, \tau_j) \right) \geq \frac{3 \epsilon_0 - \epsilon_0}{2} = \epsilon_0.
\]

It is easy to see (cf. also \cite[Theorem 6.1]{R70})
\[
\frac{\tau_0 + \tau_j}{2} \longrightarrow \frac{\tau_0 + \tau_*}{2} \in \Omega.
\]
Using the continuity of $h$ again, we get
\[
h(\xx_0, \frac{\tau_0 + \tau_*}{2}) = \lim_{j \longrightarrow +\infty} h(\xx_j, \frac{\tau_0 + \tau_j}{2}) \geq \epsilon_0 > 0 = H(\xx_0) = \sup_{\tau \in \Omega} h(\xx_0, \tau),
\]
which leads to a contradiction.

This completes the proof of Proposition \ref{NPA}.
\end{proof}

\medskip

{\bf Proof of Theorem \ref{THN6}}:

\medskip

Recall $d$ is continuous w.r.t. the usual Euclidean topology (cf. \cite{VSC92}), applying Proposition \ref{NPA} to the function $\phi(g; \tau)$ with $\OA$, we obtain $d(g)^2 = \sup_{\tau \in \OA} \phi(g; \tau)$ for all $g \in \overline{\widetilde{\M}}$.  \qed

\subsection{Proof of Theorem \ref{aT121}}

\begin{proof}
We argue by contradiction: assume that there exists $g_0 \in \partial \M$ such that $g_0 \not\in \mathrm{Cut}_o$.

First, for any $g \not\in \mathrm{Cut}_o$, the proximal sub-differential of the squared sub-Riemannian distance at the point $g$, $\partial_P d(g)^2$, is the singleton
\[
\partial_P d(g)^2 = \left\{ (\nabla_x d(g)^2, \nabla_t d(g)^2)  \right\} \ \mbox{with $\nabla$ the usual gradient in Euclidean space,}
\]
since $d^2$ is $C^{\infty}$ in a neighborhood of $g$ (see \cite{RT05} and \cite[p.\! 36]{CLSW98}). Let $\gamma_g(s)$ ($0 \leq s \leq 1$) denote the unique shortest geodesic joining $o$ to $g$, and $(\gamma_g(s), (\xi_g(s), \tau_g(s)))$ its unique normal  extremal lift. By \cite[Proposition 2]{RT05} (see also \cite[\S 11.1]{ABB20}), the  final covector $(\xi_g(1), \tau_g(1))$ equals $\frac{1}{2} \partial_P d(g)^2$. Moreover, \eqref{HaEP} implies that
\begin{align} \label{n1n}
\frac{1}{4} \nabla_t d(g)^2 = \frac{1}{2} \tau_g(1) \equiv \frac{1}{2} \tau_g(s) (0 \leq s \leq 1) \in \OA, \qquad \forall g \in \M.
\end{align}

Recall that $\OM_2 \subseteq \mathrm{Abn}_o^* \subseteq \mathrm{Cut}_o$ (see Corollary \ref{nCCn}), then $g_0 \in \overline{\OM} \setminus \OM$. From (1) of Remark \ref{Nr27}, there exists $\varsigma_0 = (\zeta(0), \tau(0) = 2 \theta_0) \in \R^q \times \R^m$ such that
\[
\theta_0 \in \partial \OA, \quad  |\zeta(0)| = d(g_0), \quad g_0 = \exp{\varsigma_0}.
\]

Under our assumption that $\overline{\M} = \G$, it follows from Lemma \ref{AxL} that in a neighborhood of $\varsigma_0$, there exists $(\zeta_*, 2 \theta_*)$ with $\theta_* \not\in \OA$ such that $g_* = \exp\{(\zeta_*, 2 \theta_*)\} \in \M \subset \mathcal{S}$ and the unique shortest geodesic from $o$ to $g_*$ is $\gamma_{g_*}(s) = \exp\{s \, (\zeta_*, 2 \theta_*)\}$ ($0 \leq s \leq 1$). Recall that $\gamma_{g_*}$ is not abnormal. Consider the unique normal  extremal lift of $\gamma_{g_*}$, its final covector $(\xi_{g_*}(1), \tau_{g_*}(1))$ satisfies $\tau_{g_*}(1)/2 = \theta_* \not\in \OA$ which is in contradiction with \eqref{n1n}.
\end{proof}

\medskip

\renewcommand{\theequation}{\thesection.\arabic{equation}}
\section{Application 1: Generalized H-type groups}  \label{ghg}
\setcounter{equation}{0}

\medskip

Recall that $f(s) := 1 - s \cot{s}$ and  $\mu(s) := f'(s)$.

We generalize the known results in this section. A step-2 group $\G(q, m, \U)$ is of generalized Heisenberg-type if there exists a positive definite $q \times q$ real matrix $S$ such that
\begin{align} \label{GHd1}
U(\lambda) U(\lambda') + U(\lambda') U(\lambda) = 2 \lambda \cdot \lambda' S^2, \qquad \forall \lambda, \lambda' \in \R^m.
\end{align}
It is well-known that $q = 2 n$ is even. Without loss of generality, we may assume that
\begin{eqnarray} \label{GHd2}
S = \left(
	\begin{array}{ccc}
	\widetilde{a}_1 \, \mathbb{I}_{2 k_1}& \ & \ \\
	\  & \ddots & \ \\
	\ & \ & \widetilde{a}_{l_0} \, \mathbb{I}_{2 k_{l_0}}\\
	\end{array}
	\right)
\end{eqnarray}
with
\begin{align} \label{GHd3}
0 < \widetilde{a}_1 < \cdots < \widetilde{a}_{l_0} = 1, \qquad 2 n = \sum_{j = 1}^{l_0} 2 k_j.
\end{align}
Otherwise, it suffices to use a linear change of basis on $\R^{2 n} \times \R^m$
\[
(x, t) \longmapsto (O \cdot x, r \, t)
\]
with a suitable $(2 n) \times (2 n)$ orthogonal matrix $O$ and $r \in \R \setminus\! \{0\}$.

Recall that the case $m = 1$ corresponds to generalized Heisenberg groups. Moreover the case $S = \mathbb{I}_{2 n}$, denoted by $\H$, is classical Heisenberg-type groups (H-type group, in short). Notice that there exists an H-type group of type $(2 n, m)$, $\H$, if and only if $m < \rho(2n)$, where the Hurwitz-Radon function $\rho$ is defined by
\begin{align} \label{RMN}
\rho(2n) := 8 k + 2^l, \quad \mbox{where } 2 n = \mbox{(odd)} \cdot 2^{4k + l}, \ 0 \leq l \leq 3, \ k = 0, 1, \ldots,
\end{align}
see for example \cite{K80} or \cite[\S 3.6 and \S 18]{BLU07}.

Thus for general $S$, we have $m < \min_{1 \le j \le l_0} \rho(2 k_j)$. Let us write in this section
\begin{align*}
x = (\x_1, \ldots, \x_{l_0}) \in \R^{2 k_1} \times \cdots \times \R^{2 k_{l_0}}.
\end{align*}
A simple calculation implies that
\begin{align} \label{N62}
U(\lambda) \cot{U(\lambda)} = \left(
	\begin{array}{ccc}
	\widetilde{a}_1 |\lambda| \cot{(\widetilde{a}_1 |\lambda|)} \, \mathbb{I}_{2 k_1}& \ & \ \\
	\  & \ddots & \ \\
	\ & \ & \widetilde{a}_{l_0} |\lambda| \cot{(\widetilde{a}_{l_0} |\lambda|)} \, \mathbb{I}_{2 k_{l_0}}\\
	\end{array}
	\right), \nonumber \\
\OA = B_{\R^m}(0, \pi) \mbox{ \  and \  } \phi((x, t); \tau) = 4 \, t \cdot \tau + \sum_{j = 1}^{l_0} |\x_j|^2 \, \a_j |\tau| \cot{(\a_j |\tau|)}.
\end{align}

In this case, a direct calculation shows that for any $x \neq 0$, $\phi((x, t); \tau)$ is strongly concave in $\OA$. Then the Hessian matrix of $\phi((x, t); \cdot)$ at any $\tau \in \OA$ is negative-definite, and the set $\M$ defined by \eqref{m} is in fact
\begin{align*}
\left\{ (x, t); x \neq 0, \exists \theta \in \OA \mbox{ such that } 4 t = - \nabla_{\theta} \langle U(\theta) \cot{U(\theta)} \, x, x  \rangle \right\}.
\end{align*}
Moreover, for $x \neq 0$, the proper concave function in $\overline{\OA}$, $\phi((x, t); \tau)$ has a unique maximizer because $\overline{B_{\R^m}(0, \pi)}$ is strictly convex. Indeed, these results are still valid in the general setting of M\'etivier groups, see Theorem \ref{nTH73} below.

Using the explicit expression of $\phi$, we can depict $\M$. Furthermore, Theorems \ref{TH2} and \ref{THN6} allow us to obtain the exact formulas for $d^2$. More precisely, we have the following:

\begin{theo} \label{nNTH61}
(1) We have
{\em\begin{align} \label{DoM}
\M = \left\{ (x, t); \ \x_{l_0} \neq 0, \mbox{ or $\x_{l_0} = 0$ with }
4 |t| < \sum_{j = 1}^{l_0 - 1} \a_j |\x_j|^2 \mu(\a_j \pi) \right\}.
\end{align}}
For $g = (x, t) \in \M$, there exists a unique $\theta = \theta(g) \in B_{\R^m}(0, \pi)$ such that
{\em\begin{align} \label{ECP}
4 t = \theta \, \sum_{j = 1}^{l_0} \a_j^2 \,  |\x_j|^2 \, \frac{\mu(\a_j |\theta|)}{\a_j |\theta|}, \  \mbox{ i.e. } \  \theta = 4 \, \left( \sum_{j = 1}^{l_0} \a_j^2 \,  |\x_j|^2 \, \frac{\mu(\a_j |\theta|)}{\a_j |\theta|} \right)^{-1} t.
\end{align}}
Moreover,
{\em \begin{align} \label{GHD1}
d(g)^2 = 4 |t| \, |\theta| + \sum_{j = 1}^{l_0} |\x_j|^2 \, \a_j |\theta| \cot{(\a_j |\theta|)} = \sum_{j = 1}^{l_0} |\x_j|^2 \left( \frac{\a_j |\theta|}{ \sin{(\a_j |\theta|)}} \right)^2 = \left| \frac{U(\theta)}{\sin{U(\theta)}} x \right|^2.
\end{align}}

(2) The cut locus of $o$, $\mathrm{Cut}_o$, is
\[
\M^c = \left\{ (x, t); \ \x_{l_0} = 0,
4 |t| \geq \sum_{j = 1}^{l_0 - 1} \a_j |\x_j|^2 \mu(\a_j \pi) \right\}.
\]
If $(x, t) \in \M^c$, then
{\em\begin{align} \label{GHD2}
d(x, t)^2 = \pi \left( 4  |t| + \sum_{j = 1}^{l_0 - 1} \a_j |\x_j|^2 \cot{(\a_j \pi)} \right).
\end{align}}

(3) For any $g \neq o$, there exists exactly one $\theta \in \overline{\OA}$ such that
\[
d(g)^2 =  \phi(g; \theta) = \sup_{\tau \in \OA} \phi(g; \tau).
\]
\end{theo}

\begin{proof} (1) follows from the fact that the stationary point satisfies the following equation
\begin{align*}
0 = \nabla_{\theta} \phi(x, t; \theta) = 4 t - \theta \, \sum_{j = 1}^{l_0} \a_j^2 \, |\x_j|^2 \, \frac{\mu(\a_j |\theta|)}{\a_j |\theta|}.
\end{align*}
And other results are clear because of Theorems \ref{THN6} and \ref{aT121}.
\end{proof}

\begin{remark}
(1) We can obtain shortest geodesic(s) as well as normal geodesics joining $o$ to any fixed $g \neq o$. This is left to the interested reader.

(2) In the case of generalized Heisenberg groups, namely $m = 1$, the main results of Theorem \ref{nNTH61} can be found in \cite{BGG00}.
\end{remark}

Let us recall some facts: in the special case of general Heisenberg groups (namely, $m = 1$) and of H-type groups (i.e. $S = \mathbb{I}_q$), the Carnot-Carath\'eodory distances (also the cut locus and the shortest geodesics) are obtained, for example in \cite{G77, BGG00, TY04} and \cite{R05}, by using a classical method from control theory. And our approach is completely different. In the next section, we adapt our method to study the sub-Riemannian distance on more general case: M\'etivier groups. To finish this section, we explain this fact that

\subsection{For any given $\G(q, m, \U)$, there is an uncountable number of GM-groups $\G(q + 2 n, m, \widetilde{\U})$} \label{sNss81}

Indeed, for a fixed $\G(q, m, \U)$ with $\U = \left\{ U^{(1)}, \ldots, U^{(m)} \right\}$. Taking $\epsilon_j > 0$ ($1 \le j \le m$) such that $\sum_{j = 1}^m \epsilon_j^2 \, \| U^{(j)} \|^2 < 1$, and a generalized Heisenberg-type group $\G(2 n, m, \U_*)$ defined by \eqref{GHd1}-\eqref{GHd3} with $\U_* = \left\{ U_*^{(1)}, \ldots, U_*^{(m)} \right\}$, one consider $\G(q + 2 n, m, \widetilde{\U})$ where
\begin{align*}
\widetilde{\U} := \left\{ \u^{(1)}, \ldots, \u^{(m)} \right\}, \qquad \u^{(j)} := \left(
  \begin{array}{cc}
    \epsilon_j \, U^{(j)} & \mbox{} \\
    \mbox{} & U_*^{(j)} \\
  \end{array}
\right).
\end{align*}

Let $g = (x', (\x_1, \ldots, \x_{l_0}), t)$ with $x' \in \R^q$, $t \in \R^m$ and $(\x_1, \ldots, \x_{l_0}) \in \R^{2 k_1} \times \cdots \times \R^{2 k_{l_0}} = \R^{2n}$. Notice that $\| ( i \sum_{j = 1}^m \lambda_j \u^{(j)})^2 \| = |\lambda|^2$, and $\G(q + 2 n, m, \widetilde{\U})$ is of type GM because of $g \in \M$ whenever $\x_{l_0} \neq 0$ by Proposition \ref{nPr210}.

\medskip

\renewcommand{\theequation}{\thesection.\arabic{equation}}
\section{Application 2: M\'etivier groups}  \label{MG}
\setcounter{equation}{0}

\medskip

Recall that $\G(q, m, \U)$ is of M\'etivier if $U(\lambda)$ is invertible for any $\lambda \in \R^m \setminus \{ 0 \}$  (cf. \cite[Hypoth\`ese (H)]{M80}, namely, nonsingular in the sense of \cite{E94}). Notice that all generalized H-type groups are of M\'etivier. As in Subsection \ref{sNss81}, for any given M\'etivier group $\G(q, m, \U)$, there is an uncountable number of GM-groups of M\'etivier  $\G(q + 2 n, m, \widetilde{\U})$. In this section, we do not care about the explicit expression of $\OA$ defined by \eqref{oa}.

Recall that $f(s) := 1 - s \cot{s}$ (see \eqref{FF}),  $h_{\gamma}(s) := (1 - \gamma \, s)^{-1}$ (see \eqref{HDG}). Also for $\tau \in \R^m$, $g = (x, t)$,
\begin{align*}
\widetilde{U}(\lambda) := \sum_{j = 1}^m \lambda_j U^{(j)},  \quad U(\tau) := i \sum_{j = 1}^m \tau_j U^{(j)},\qquad \phi(g; \tau) := \langle U(\tau) \cot{U(\tau)} \, x, x  \rangle + 4 t \cdot \tau,
\end{align*}
defined in  \eqref{2c'} and \eqref{2cn} respectively.
Let us begin with the following:

\begin{lem} \label{nLm73}
If $\G$ is of M\'etivier, then the set $\overline{\OA}$ is strictly convex, namely, for any $\tau \neq \tau' \in \overline{\OA}$ and $0 < s < 1$, we have $s \tau + (1 - s) \tau' \in \OA$.
\end{lem}

\begin{proof}
It suffices to show that for $\tau \neq \tau' \in \partial \OA$ and $0 < s < 1$, we have
\[
\| s \, \widetilde{U}(\tau) + (1 - s) \, \widetilde{U}(\tau') \| < \pi.
\]

Suppose that the assertion claimed is false. Then there exist $\tau_o \neq \tau'_o$ and $0 < s_0 < 1$ such that
\[
\| \widetilde{U}(\tau_o) \| = \| \widetilde{U}(\tau'_o) \| = \| s_0 \, \widetilde{U}(\tau_o) + (1 - s_0) \, \widetilde{U}(\tau'_o) \| = \pi.
\]

Therefore, there exists a unit vector $\xi$ so that
$\| s_0 \, \widetilde{U}(\tau_o) \xi + (1 - s_0) \, \widetilde{U}(\tau'_o) \xi \|^2 = \pi^2$.
Set $a = \widetilde{U}(\tau_o) \xi$ and $b = \widetilde{U}(\tau'_o) \xi$. We get $|a|, |b| \leq \pi$ and $| s_0 a + (1 - s_0) b| = \pi$ with $0 < s_0 < 1$. These facts imply that $a = b$ and $|a| = \pi$. Namely $U(\tau_o) \xi = U(\tau'_o) \xi$, i.e. $U(\tau_o - \tau'_o) \xi = 0$. Then $U(\tau_o - \tau'_o)$ is singular which contradicts the definition of M\'etivier groups.
\end{proof}

\begin{remark}
The result of Lemma \ref{nLm73} is obviously no longer valid for general $2$-step groups.
\end{remark}

In the setting of M\'etivier groups, Proposition \ref{TH1} can be improved to the following:

\begin{prop} \label{P71}
Let $g = (x, t)$ with $x \neq 0$. Then $\phi(g; \cdot)$ is strongly concave in $\OA$.
\end{prop}

\begin{proof}
It suffices to show that $\langle f(U(\tau)) x, x \rangle$ is strongly convex in $\OA$.

By \eqref{IS1} and \eqref{IS}, we can write
$f = f^* + h_{\frac{1}{2 \pi}} + h_{- \frac{1}{2 \pi}}$ with
\begin{align*}
f^*(s) = -2 + \int_{\lambda \in [-\frac{1}{\pi}, \, \frac{1}{\pi}] \setminus \{\pm \frac{1}{2 \pi} \}} \frac{s^2}{1 - \lambda s} \, d\nu(\lambda). \end{align*}

It follows from Loewner's theorem that $f^*$ is operator convex on $(-\pi, \, \pi)$, so the choice of $\OA$ implies that $\langle f^*(U(\tau)) x, x \rangle$ is convex in $\OA$. Then it remains to show that $\langle h_{\frac{1}{2 \pi}}(U(\tau)) x, x \rangle$ and $\langle h_{- \frac{1}{2 \pi}}(U(\tau)) x, x \rangle$
are strongly convex  in $\OA$, namely, there exists a constant $c = c(x) > 0$ such that
\begin{gather*}
\left\langle \left( s h_{\pm \frac{1}{2 \pi}}(U(\tau_o))  + (1 - s)  h_{\pm \frac{1}{2 \pi}}(U(\tau'_o)) - h_{\pm \frac{1}{2 \pi}}(U(s \tau_o + (1 - s) \tau'_o)) \right) x, x \right\rangle  \\
:= \widetilde{h}_x^{(\pm)}(\tau_o, \tau'_o, s) \geq c \, s \, (1 - s) |\tau_o - \tau'_o|^2, \quad \forall 0 < s < 1, \  \tau_o \neq \tau'_o \in \OA.
\end{gather*}

We prove the above inequality only for the case $h_{\frac{1}{2 \pi}}$, since the other case $h_{-\frac{1}{2 \pi}}$ follows similarly. Notice that for any $q \times q$ invertible matrices $A$ and $B$, we have  (see for example \cite[(6.6.60), p. 555]{HJ91})
\begin{gather*}
s A^{-1} + (1 - s) B^{-1} - \left( s A + (1 - s) B \right)^{-1} \\
= s (1 - s) A^{-1} \left( B - A \right) B^{-1} \left( s B^{-1} + (1 - s) A^{-1} \right)^{-1} A^{-1} \left( B - A \right) B^{-1}.
\end{gather*}
Take now $A = \mathbb{I}_q - \frac{1}{2 \pi} U(\tau_o)$ and $B = \mathbb{I}_q - \frac{1}{2 \pi} U(\tau'_o)$ in the last formula. Using again the definition of $\OA$, we find that $A^{-1}$, $B^{-1}$ and $\left( s B^{-1} + (1 - s) A^{-1} \right)^{-1}$ are all positive definite with spectra in $[10^{-1}, \, 10]$. Observe that $A^{-1} \left( B - A \right) B^{-1} = A^{-1} - B^{-1}$ is Hermitian, then
\begin{align*}
 \widetilde{h}_x^{(+)}(\tau_o, \tau'_o, s) &\geq \frac{s \, (1 - s)}{10} | A^{-1} \left( B - A \right) B^{-1} x|^2 \\
 &=  \frac{s \, (1 - s)}{40 \pi^2} \, |\tau_o - \tau'_o|^2 \, \left| A^{-1} U(\frac{\tau_o - \tau'_o}{|\tau_o - \tau'_o|}) B^{-1} x \right|^2
\end{align*}
since $B - A = \frac{|\tau_o - \tau'_o|}{2 \pi} U(\frac{\tau_o - \tau'_o}{|\tau_o - \tau'_o|})$. Thus it suffices to show that there exists a constant $c(x) > 0$ such that
\begin{align*}
 \left| A^{-1} U(\frac{\tau_o - \tau'_o}{|\tau_o - \tau'_o|}) B^{-1} x \right| \geq c(x).
\end{align*}

Indeed, for a $q \times q$ Hermitian matrix $\T$, we set
\begin{align*}
\widetilde{\lambda}_{\T} = \inf \left\{ |\lambda|; \lambda \mbox{ is an eigenvalue of } \T \right\}.
\end{align*}
Then
\begin{align*}
 \left| A^{-1} U(\frac{\tau_o - \tau'_o}{|\tau_o - \tau'_o|}) B^{-1} x \right| \geq \widetilde{\lambda}_{A^{-1}}  \widetilde{\lambda}_{B^{-1}} |x| \inf_{\tau \in \mathbb{S}^{m - 1}} \widetilde{\lambda}_{U(\tau)} \geq \frac{|x|}{100}  \inf_{\tau \in \mathbb{S}^{m - 1}} \widetilde{\lambda}_{U(\tau)} > 0,
\end{align*}
where we used in the last inequality the fact that the function $\widetilde{\lambda}_{U(\tau)} > 0$ is continuous on the compact set $\mathbb{S}^{m - 1}$.
\end{proof}

By Lemma \ref{nLm73} and Proposition \ref{P71}, the following result is a direct consequence of Theorems \ref{TH2} and \ref{THN6}:

\begin{theo} \label{nTH73}
Let $\G$ be of M\'etivier, then:

(1) The set $\M$ defined by \eqref{m} can be characterized as
\begin{align*}
\M = \left\{  g = \left(x, -\frac{1}{4} \nabla_{\theta} \langle U(\theta) \cot{U(\theta)} \, x, x  \rangle\right); \ x \neq 0, \  \theta \in \OA \right\}.
\end{align*}

(2) For any $g = (x, t)$ with $x \neq 0$, there exists exactly one $\theta \in \overline{\OA}$ such that
\[
\phi(g; \theta) = \sup_{\tau \in \OA} \phi(g; \tau).
\]
Moreover, we have $d(g)^2 = \max_{\tau \in \OA} \phi(g; \tau)$ for $g \in \M$, and $d(g)^2 = \sup_{\tau \in \OA} \phi(g; \tau)$ for $g \in \overline{\M}$.
\end{theo}

\begin{remark}
We can find in \cite{LZ20} some M\'etivier groups that do not satisfy $\overline{\M} = \G$. However, we will provide, in Section \ref{N32} below, a method to determine the squared distance on general step-two groups.
\end{remark}

\medskip

\renewcommand{\theequation}{\thesection.\arabic{equation}}
\section{Application 3: Star graphs $K_{1, n}$ ($n \geq 2$)}  \label{star}
\setcounter{equation}{0}

\medskip

Recall that $f(s) := 1 - s \cot{s}$ and $\psi(s) := \frac{f(s)}{s^2}$.

The star graphs $K_{1, n}$ are prototypes of Heisenberg-like but not M\'etivier Lie groups, see for example \cite{GM00}, \cite{DDM11} and \cite{DCDMM18}.
Let $K_{1, n} = (\left\{ \X_j \right\}_{j = 1}^{n +1}, \left\{  \Y_k \right\}_{k = 1}^n) \cong \R^{n + 1} \times \R^n$. The Lie algebra is defined by the brackets $[\X_1, \, \X_{j + 1}] = -\Y_j$ for $j = 1, \dots, n$, and $0$ for others. Hence, for a column vector $\tau  \in \R^n$, we have
\begin{align*}
U(\tau) = i \left(
                \begin{array}{cc}
                  0 & \tau^T \\
                  -\tau & \mathbb{O}_{n \times n}
                \end{array}
              \right),
\end{align*}
where $\mathbb{O}_{q_1 \times q_2}$ denotes the $q_1 \times q_2$ null matrix. As a consequence, we get
\begin{align*}
U(\tau)^2 = \left(
                \begin{array}{cc}
                  |\tau|^2 & \mathbb{O}_{1 \times n} \\
                  \mathbb{O}_{n \times 1} & \tau \, \tau^T
                \end{array}
              \right),
\qquad \OA = B_{\R^n}(0, \pi).
\end{align*}
Let $x = (x_1, x_*) \in \R \times \R^n$, a simple calculation yields the identity
\begin{align} \label{star1}
\phi((x, t); \tau) = |x|^2 + 4 t \cdot \tau - \left[ x_1^2 \, f(|\tau|) + \langle \tau, x_* \rangle^2 \, \psi(|\tau|) \right].
\end{align}

Let $\M$ be the set defined in \eqref{m}, our main result is the following:

\begin{theo} \label{Kn1T}
(1) For $g \in K_{1, n}$, we have $d(g)^2 = \sup_{\tau \in \OA} \phi(g; \tau)$.

(2) The cut locus of $o$, $\mathrm{Cut}_o$,
is
\[
\M^c = \left\{((0, x_*), t);  \ |x_* \cdot t| \leq \frac{|x_*|^2}{\sqrt{\pi}} \sqrt{\left|  t - \langle t, \frac{x_*}{|x_*|} \rangle  \frac{x_*}{|x_*|} \right|}  \right\},
\]
with the understanding $\frac{x_*}{|x_*|} = 0$ if $x_* = 0$. And $d(0, t)^2 = 4 \pi |t|$. Moreover for $((0, x_*), t) \in \M^c$ with $x_* \neq 0$, we have
{\em\begin{align} \label{nDE*}
d((0, x_*), t)^2 = |x_*|^2 + 4 \pi \left|  t - \langle t, \frac{x_*}{|x_*|} \rangle  \frac{x_*}{|x_*|} \right|.
\end{align}}

(3) If $(x, t) = ((x_1, x_*), t) \in \M$, then there exists a unique $\theta = \theta(x, t) \in \OA$ such that
{\em \begin{align} \label{n1KE}
t &= \frac{1}{4} \nabla_{\theta} \left[ f(|\theta|) \, x_1^2 + \psi(|\theta|) \langle \theta, x_* \rangle^2 \right] \nonumber \\
&= \frac{1}{4} \left[ \frac{f'(|\theta|)}{|\theta|} \, x_1^2 \, \theta  + \frac{\psi'(|\theta|)}{|\theta|} \, \langle \theta, x_* \rangle^2 \, \theta + 2 \, \psi(|\theta|) \langle \theta, x_* \rangle \, x_* \right].
\end{align}}
Furthermore, we have
{\em\begin{align} \label{n2Kd}
d(x, t)^2 &= \phi((x, t); \theta) = |x|^2 + 4 \, t \cdot \theta - \left[ x_1^2 \, f(|\theta|) + \langle \theta, x_* \rangle^2 \, \psi(|\theta|) \right] \nonumber \\
&= \left( \frac{|\theta|}{\sin{|\theta|}} \right)^2 x_1^2 + |x_*|^2 + \left[ \left( \frac{|\theta|}{\sin{|\theta|}} \right)^2 - 1 \right] \left( x_* \cdot \frac{\theta}{|\theta|} \right)^2 = \left| \frac{U(\theta)}{\sin{U(\theta)}} x \right|^2.
\end{align}}
\end{theo}

\begin{remark}
In this case, we can obtain shortest geodesic(s) as well as normal geodesics joining $o$ to any fixed $g \neq o$. See \cite{LZ20} for more details.
\end{remark}

It follows from Proposition \ref{nPr210} that $\M$ contains
\begin{align}
\M_1 := \left\{((x_1, x_*), t) \in \R^{n + 1} \times \R^n; \ x_1 \neq 0 \right\},
\end{align}
which is dense in $K_{1, n}$. Hence the first and the third assertion of Theorem \ref{Kn1T} as well as $d(0, t)^2 = 4 \pi |t|$ are direct consequence of Theorems \ref{TH2} and \ref{THN6}. From Theorem \ref{aT121}, it remains to prove that $\M = \M_1 \cup \M_2$ where
\begin{align}
\M_2 := \left\{((0, x_*), t); \ x_* \neq 0, \, |x_* \cdot t| > \frac{|x_*|^2}{\sqrt{\pi}} \sqrt{\left|  t - \langle t, \frac{x_*}{|x_*|} \rangle  \frac{x_*}{|x_*|} \right|} \right\},
\end{align}
and \eqref{nDE*}. Let us begin by the following simple observation:

\begin{lem} \label{nKL1}
For any $((x_1, x_*), t) \in K_{1, n}$ and any $n \times n$ orthogonal matrix $O$, we have
\begin{align*}
d((x_1, x_*), t) = d((|x_1|, O \, x_*), O \, t).
\end{align*}
Moreover, set $e_1 = (1, 0, \ldots, 0)$ and $e_2 \in \R^n$ with $1$ in the second coordinate and zeros elsewhere, we have
\[
d((|x_1|, |x_*| \, e_1), t_1 \, e_1 + t_2 \, e_2) = d((|x_1|, |x_*| \, e_1), |t_1| \, e_1 + t_2 \, e_2).
\]
\end{lem}

\begin{proof}
In the setting of $K_{1, n}$, by means of the explicit expression for the heat kernel (see \eqref{2c0}), a simple orthogonal transformation shows that
\begin{gather*}
p((x_1, x_*), t) = p((|x_1|, O \, x_*), O \, t), \\
p((|x_1|, |x_*| \, e_1), t_1 \, e_1 + t_2 \, e_2) = p((|x_1|, |x_*| \, e_1), |t_1| \, e_1 + t_2 \, e_2).
\end{gather*}
Then the scaling property (see \eqref{sp}) and Varadhan's formulas (see \eqref{VF}) imply the desired results.
\end{proof}

Thus, we may assume in the sequel, when we need to, that
\begin{align*}
x = \left( |x_1|, |x_*| \, (1, 0, \ldots, 0)  \right) = (|x_1|, |x_*| \, e_1) , \qquad t = t_1 \, e_1 + t_2 \, e_2 \mbox{ with } t_1, \ t_2 \geq 0.
\end{align*}

\medskip

\subsection{Proof of $\M = \M_1 \cup \M_2$}

\medskip

We begin by considering the case where $x_1 \neq 0$. Under our assumption, Proposition \ref{TH1} can be improved to the following:

\begin{prop} \label{StarP}
Let $x_1 \neq 0$. Then the function $\phi((x, t); \cdot)$ is strongly concave in $\OA$.
\end{prop}

\begin{proof}
It suffices to show that for any $r := |\tau| < \pi$, we have $\He f(|\tau|) \ge 2 \psi(0)$ and $\He \left( \tau_1^2 \, \psi(|\tau|) \right) \ge 0$.

A simple calculus gives
\begin{align*}
\He f(r) = \frac{f'(r)}{r} \, \mathbb{I}_n + r^{-2} \, \left( f^{''}(r) - \frac{f'(r)}{r} \right) \, \tau \, \tau^T \geq 2 \psi(0) \, \mathbb{I}_n,
\end{align*}
by \eqref{Ii1}.

Moreover, using the following elementary property of Hessian
\begin{align*}
\He(f_1 f_2) = f_1 \, \He f_2 + f_2 \, \He f_1 + \left( \nabla f_1 \right)^T \nabla f_2 + \left( \nabla f_2 \right)^T \nabla f_1,
\end{align*}
$\He \left( \tau_1^2 \, \psi(|\tau|) \right)$ is given by
\begin{align*}
&\tau_1^2 \,  \left[ \frac{\psi'(r)}{r} \mathbb{I}_n + \left( \psi^{''}(r) - \frac{\psi'(r)}{r} \right) \frac{\tau}{r} \left( \frac{\tau}{r} \right)^T \right] + 2 \psi'(r) \frac{\tau}{r} (\tau_1, \mathbb{O}_{1 \times (n - 1)}) \\
&+ 2 \psi'(r) \left(
                    \begin{array}{c}
                      \tau_1 \\
                      \mathbb{O}_{(n - 1) \times 1} \\
                    \end{array}
                  \right)
\left( \frac{\tau}{r} \right)^T + 2 \psi(r) \left(
\begin{array}{cc}
1 & \mathbb{O}_{1 \times (n - 1)} \\
\mathbb{O}_{(n - 1) \times 1} & \mathbb{O}_{(n - 1) \times (n - 1)} \\
\end{array}
\right).
\end{align*}

It remains to show that for any $\tau = (\tau_1, \ldots, \tau_n) \in \OA$ and any $u = (u_1, \ldots, u_n) \in \R^n$ satisfying $|u| = 1$, we have
\begin{align} \label{stare}
0 &\le \langle \He \left( \tau_1^2 \, \psi(|\tau|) \right) \, u, u \rangle \nonumber \\
&= \left[ \frac{\psi'(r)}{r} + \left( \psi^{''}(r) - \frac{\psi'(r)}{r} \right) s^2 \right] \tau_1^2 + 4 \psi'(r) \, s \, u_1 \, \tau_1 + 2 \psi(r) \, u_1^2
\end{align}
where $s = \langle u, \frac{\tau}{r} \rangle \in [-1, 1]$. We consider the right-hand side of \eqref{stare} as a quadratic polynomial w.r.t. $\tau_1$. \eqref{Ii2} implies that the coefficient of the first term is strictly positive. On the other hand, its discriminant
\begin{align*}
&8 \left\{ 2 \left( \psi'(r) \, s \, u_1 \right)^2 - \psi(r)  \, u_1^2 \left[ \frac{\psi'(r)}{r} + \left( \psi^{''}(r) - \frac{\psi'(r)}{r} \right) s^2  \right] \right\} \\
&= - 8 \, u_1^2  \left\{ \psi(r) \frac{\psi'(r)}{r} (1 - s^2) + \left( \psi(r) \, \psi^{''}(r) - 2 \psi'(r)^2 \right) s^2 \right\}
\end{align*}
is non-positive by Lemma \ref{n32l} (remark also that it is strictly negative if $u_1 \neq 0$), which completes the proof of \eqref{stare}.

This proves the proposition.
\end{proof}

Recall that $\OM_2$ (resp. $\OM$) is defined by \eqref{OM2} (resp. \eqref{nOM}). For $g = ((0, |x_*| \, e_1), t_1 \, e_1 + t_2 \, e_2)$ with $t_1$, $t_2 \geq 0$, the reference function (cf. \eqref{star1}) can be simplified as
\begin{align} \label{nnKSnn1}
\phi(g; \tau) = |x_*|^2 + 4 \, (t_1 \tau_1 + t_2 \tau_2) - |x_*|^2 \, \tau_1^2 \, \psi(|\tau|), \qquad \tau \in \OA = B_{\R^n}(0, \pi).
\end{align}
We have the following simple observation:

\begin{lem} \label{nSTARnL}
It holds that $g = ((0, |x_*| \, e_1), 0) \in \OM_2$. Moreover, $g \not\in \OM$ in the case where: (1) $|x_*| = 0$ with $t \neq 0$; or (2) $|x_*| \neq 0$ with $t_1 = 0$ and $t_2 \neq 0$.
\end{lem}

\begin{proof}
The proof is easy. Indeed, for $g = ((0, |x_*| \, e_1), 0)$, we have $\phi(g; e_2) = \phi(g; - e_2) = \max_{\tau \in \OA} \phi(g; \tau)$. Furthermore, if $|x_*| = 0$ with $t \neq 0$ (resp. $|x_*| \neq 0$ with $t_1 = 0$ and $t_2 \neq 0$), $\phi(g; \cdot)$ has a unique maximum point $\pi \frac{t}{|t|}$ (resp. $\pi \, e_2$) which belongs to $\partial \OA$.
\end{proof}

Next, it follows from the proof of Proposition \ref{StarP} that $\He \left( \tau_1^2 \, \psi(|\tau|) \right) > 0$ for $\tau_1 \neq 0$. Hence we have, via Theorem \ref{TH3N}, the following:

\begin{prop} \label{NRk84}
Suppose that $g = ((0, |x_*| \, e_1), t_1 \, e_1 + t_2 \, e_2)$ with $|x_*| \cdot t_1 > 0$ and $t_2 \geq 0$. Then $g \in \M$ if and only if one of the following two conditions holds:

(1) $t_2 > 0$ and there exists $\tau \in \Omega_{+ +}$ with
\begin{align}
\Omega_{+ +} := \left\{(\tau_1, \tau_2); \  \tau_1, \tau_2 >0, \sqrt{\tau_1^2 + \tau_2^2} < \pi \right\}
 \end{align}
such that
\begin{align} \label{NSEo}
4 (t_1, t_2) = |x_*|^2 \, \Upsilon(\tau_1, \tau_2),
\end{align}
where the smooth function $\Upsilon$ is defined by
\begin{align} \label{nSEI}
\Upsilon(\tau_1, \tau_2) := \nabla_{\R^2} \left( \tau_1^2 \psi(|\tau|) \right) = \tau_1 \left( 2 \psi(|\tau|) + \psi'(|\tau|) \frac{\tau_1^2}{|\tau|}, \psi'(|\tau|) \frac{\tau_1 \tau_2}{|\tau|} \right).
\end{align}
Moreover, the Jacobian matrix of $\Upsilon$ is positive definite.

(2) $t_2 = 0$. In this case, we have $0 < \tau_1 < \pi$ and $\tau_2 = 0$ in \eqref{NSEo} and \eqref{nSEI}.

\end{prop}

Consider the simply connected domains in $\R^2$, $\Omega_{+ +}$ and
\begin{align}
\R^2_{+, +, >} := \left\{ (u_1, u_2); u_1, u_2 > 0, \  u_1 > \frac{2}{\sqrt{\pi}} \sqrt{u_2} \right\}.
\end{align}
We have the following

\begin{prop} \label{Kp1}
$\Upsilon$, defined by \eqref{nSEI}, is a $C^{\infty}$-diffeomorphism from $\Omega_{+ +}$ onto $\R^2_{+, +, >}$.
\end{prop}

\begin{proof}
First we claim that the smooth function $\Upsilon$ is from $\Omega_{+ +}$ into $\R^2_{+, +, >}$, namely,
\begin{align*}
 \tau_1 \left( 2 \psi(|\tau|) + \psi'(|\tau|) \frac{\tau_1^2}{|\tau|} \right) > \frac{2}{\sqrt{\pi}} \sqrt{\psi'(|\tau|) \frac{\tau_1^2}{|\tau|}  \tau_2}, \quad (\tau_1, \tau_2) \in \Omega_{+ +}.
\end{align*}
This is a direct consequence of the inequality $\psi(r) > \sqrt{\frac{\psi'(r)}{r}}$ for any $0 \leq r < \pi$, see \eqref{iN38}.

Recall that the Jacobian determinant of $\Upsilon$ vanishes nowhere. By Hadamard's theorem (see for example \cite[\S 6.2]{KP02}), it remains to show that $\Upsilon$ is proper, that is, whenever $\{ \tau^{(j)} \} \subset \Omega_{+ +}$ satisfies $\tau^{(j)} = (\tau^{(j)}_1, \tau^{(j)}_2)  \longrightarrow \partial \Omega_{+ +}$ then $\Upsilon(\tau^{(j)})  \longrightarrow \partial \R^2_{+, +, >}$. Indeed, by the definition of  $\psi$ (see \eqref{FF} and \eqref{IS0}), we have:

(1)  If  $\tau^{(j)}_1 \longrightarrow 0^+$ and $|\tau^{(j)}| \leq \pi - \varepsilon < \pi$, then $\Upsilon(\tau^{(j)})  \longrightarrow 0 \in  \partial \R^2_{+, +, >}$.

(2) If  $|\tau^{(j)}| \longrightarrow \pi^-$ and $\tau^{(j)}_1 \geq \varepsilon > 0$, by the fact that (cf. \eqref{N32ei1})
\begin{align} \label{nnp}
\lim_{r \longrightarrow \pi^-} (\pi - r) \psi(r) = \frac{1}{\pi}, \qquad \lim_{r \longrightarrow \pi^-} (\pi - r)^2 \psi'(r) = \frac{1}{\pi},
\end{align}
we have  $\Upsilon(\tau^{(j)})  \longrightarrow \infty \in  \partial \R^2_{+, +, >}$.

(3) If  $\tau^{(j)}_2 \longrightarrow 0^+$, $\tau^{(j)}_1 \geq \varepsilon > 0$ and $|\tau^{(j)}| \leq \pi - \varepsilon < \pi$, then $\Upsilon(\tau^{(j)})  \longrightarrow (0, +\infty) \times \{ 0 \} \subset  \partial \R^2_{+, +, >}$.

(4) If  $\tau^{(j)}_1 \longrightarrow 0^+$ and $|\tau^{(j)}| \longrightarrow \pi^-$, set $\Upsilon(\tau^{(j)})  = (u^{(j)}_1, u^{(j)}_2)$. Whenever $u^{(j)}_2  \longrightarrow +\infty$, then $\Upsilon(\tau^{(j)}) \longrightarrow \infty$. Assume now that  $u^{(j)}_2  \longrightarrow a \geq 0$. Then we have
\begin{align*}
a = \lim_{j \longrightarrow +\infty} u^{(j)}_2 =  \lim_{j \longrightarrow +\infty}  \frac{\tau^{(j)}_2}{|\tau^{(j)}|} \left( \tau^{(j)}_1 \right)^2 \psi'(|\tau^{(j)}|)   =  \lim_{j \longrightarrow +\infty}  \frac{1}{\pi} \left( \frac{\tau^{(j)}_1}{\pi - |\tau^{(j)}|} \right)^2,
\end{align*}
where we have used \eqref{nnp} in the last equality. Using \eqref{nnp} again, it
implies that
\begin{align*}
\lim_{j \longrightarrow +\infty} u^{(j)}_1 &=  \lim_{j \longrightarrow +\infty}  \left[ 2  \tau^{(j)}_1 \psi(|\tau^{(j)}|) + \frac{\tau^{(j)}_1}{|\tau^{(j)}|} \left( \tau^{(j)}_1 \right)^2 \psi'(|\tau^{(j)}|)  \right] \\
&=  \lim_{j \longrightarrow +\infty}  \frac{2}{\pi} \frac{\tau^{(j)}_1}{\pi - |\tau^{(j)}|}  = \frac{2}{\sqrt{\pi}} \sqrt{a}.
\end{align*}

This completes the proof.
\end{proof}

Let us define
\begin{align*}
\Omega_{(+)} := \left\{(\tau_1, \tau_2) \in B_{\R^2}(0, \pi); \  \tau_1 > 0 \right\}, \qquad \R^2_{(>)} := \left\{ (u_1, u_2); \ u_1 > \frac{2}{\sqrt{\pi}} \sqrt{|u_2|} \right\}.
\end{align*}
It follows from the proof of Proposition \ref{Kp1}, we have also

\begin{prop} \label{Kp2}
$\Upsilon$, defined by \eqref{nSEI}, is a $C^{\infty}$-diffeomorphism from $\Omega_{(+)}$ onto $\R^2_{(>)}$.
\end{prop}

\medskip

{\bf Proof of $\M = \M_1 \cup \M_2$}:

\medskip

Combining Proposition \ref{Kp2} with Proposition \ref{NRk84}, Lemma \ref{nSTARnL} and the fact that $\M_1 \subset \M$, up to an orthogonal transform, we get $\M = \M_1 \cup \M_2$.    \qed

\medskip

\subsection{Proof of \eqref{nDE*}}

\medskip

\begin{proof}
Aiming at \eqref{nDE*}, let us set
\begin{align}\label{SCn0}
\M_* := \left\{ ((0, e_1), t_1 \, e_1 + t_2 \, e_2); \  0 \leq t_1 \leq \frac{1}{\sqrt{\pi}} \sqrt{t_2} \right\}.
\end{align}
Up to an orthogonal transform and scaling, it remains to show that
\begin{align} \label{SCn1}
d((0, e_1), t_1 \, e_1 + t_2 \, e_2)^2 = 1 + 4 \pi \, t_2, \qquad \forall ((0, e_1), t_1 \, e_1 + t_2 \, e_2) \in \M_*.
\end{align}

Indeed it follows from \eqref{nnKSnn1} and the proof of Proposition \ref{Kp1} that
\[
\sup_{\tau \in \OA} \phi((0, e_1), t_1 \, e_1 + t_2 \, e_2; \tau) = \phi((0, e_1), t_1 \, e_1 + t_2 \, e_2; \pi \, e_2) = 1 + 4 \pi t_2
\]
which implies \eqref{SCn1}.
\end{proof}

\medskip

To finish this section, we provide a corollary which follows easily from the proof of Theorem \ref{Kn1T}, the strictly convexity of $\overline{\OA}$, and a simple observation:

\begin{cor}
We have $d^2((0, x_*), 0) = \phi((0, x_*), 0; \theta)$ for any $\theta \in \OA$ satisfying $\theta \cdot x_* = 0$. For $g \not\in \{((0, x_*), 0); x_* \in \R^n \}$, there exists a unique $\theta \in \overline{\OA}$ such that $d(g)^2 = \phi(g; \theta)$.
\end{cor}

\medskip

\renewcommand{\theequation}{\thesection.\arabic{equation}}
\section{Application 4:  The free $2$-step group $N_{3, 2}$}  \label{N32}
\setcounter{equation}{0}

\medskip

The purpose of this section is twofold.  On one hand, we provide an example of $\overline{\M} \subsetneqq \G$. On the other hand, we explain how to determine $d(g)^2$ for $g \in \G \setminus \overline{\M}$ (up to a possible negligible set) via our known results. See also \cite[Lemma 5]{LZ20} for a closely related result. Here we only consider the simplest case: the free Carnot group of step two and 3 generators  $N_{3, 2} = \R^3 \times \R^3$.

For column vector $\tau = (\tau_1, \tau_2, \tau_3) \in \R^3$, we set
\begin{align*}
U(\tau) = i \left(
                \begin{array}{ccc}
                  0 & \tau_3 & -\tau_2 \\
                  -\tau_3 & 0 & \tau_1 \\
                  \tau_2 & -\tau_1 & 0 \\
                \end{array}
              \right).
\end{align*}
As a consequence, we get
\begin{align*}
U(\tau)^2 = |\tau|^2 \, \mathbb{I}_3 - \tau \, \tau^T, \qquad \OA = B_{\R^3}(0, \pi).
\end{align*}
By a simple calculus, we find that
\begin{align} \label{n32na}
\frac{U(\tau)}{\sin{U(\tau)}} = \frac{|\tau|}{\sin{|\tau|}} \mathbb{I}_3 - \left( \frac{|\tau|}{\sin{|\tau|}} - 1 \right) \frac{\tau}{|\tau|} \frac{\tau^T}{|\tau|},
\end{align}
and
\begin{align} \label{n32c}
\phi((x, t); \tau) = |x|^2 \, |\tau| \cot{|\tau|} + \frac{1 - |\tau| \cot{|\tau|}}{|\tau|^2} \langle \tau, x \rangle^2  + 4 \, t \cdot \tau.
\end{align}

The following lemma is a counterpart of Lemma \ref{nKL1}:

\begin{lem} \label{Ln91}
For any $(x, t) \in N_{3, 2}$ and any $3 \times 3$ orthogonal matrix $O$, we have $d(x, t) = d(O \, x, O \, t)$. Moreover, we have $d(|x| \, e_1, t_1 \, e_1 + t_2 \, e_2) = d(|x| \, e_1, |t_1| \, e_1 + |t_2| \, e_2)$.
\end{lem}

So we may assume that $x = |x| \, (1, 0, 0) = |x| \, e_1$ and $t = (t_1, t_2, 0)$ with  $t_1, t_2 \geq 0$ if we need to. Recall that $\psi(s) := \frac{1 - s \cot{s}}{s^2}$. Under our assumption, the function $\phi$ above can be reformulated as
\begin{align*}
\phi((x, t); \tau) = |x|^2 + 4 \, t \cdot \tau - |x|^2 \, \psi_*(\tau)
\end{align*}
with
\begin{align*}
\psi_*(\tau) := \left( \tau_2^2 + \tau_3^2 \right) \, \psi(|\tau|).
\end{align*}

It is a long-standing open problem to find the exact formula for the sub-Riemannian distance on the free Carnot group of step two and $k$ ($k \geq 3$) generators $N_{k, 2} = \R^k \times \R^{\frac{k (k - 1)}{2}}$, namely the Gaveau-Brockett optimal control problem, see \cite{G77} and \cite{B82}. To our knowledge, the best known result for $k = 3$ has been obtained in \cite{MM17} where the authors solved only the special case of $t = \pm |t| \, e_1$. When $k \geq 4$, there exists fewer known result; more precisely, the exact formulas of $d^2(g)$ has been got only for $g = (0, t)$ with $t \in \R^{\frac{k (k - 1)}{2}}$ or $g = (x, 0)$ with $x \in \R^k$, see for example \cite{B82}, \cite{LS95}, \cite[\S 12.3.5]{M02} and \cite{LM14}. Now, Theorems \ref{TH2} or \ref{THN5} allow us to have some progress on the problem. In particular, we have the formulas of $d(g)^2$ at least when $g \in \widetilde{\M}$, which is a symmetric, scaling invariant set with non-empty interior. Furthermore, a necessary condition can be found in Corollary \ref{nCx}, and Corollary \ref{nC124} will be very useful if we know $\mathrm{Cut}_o$. Also recall that our main results remain valid for all $2$-step groups.

Here we will restrict our attention to the special case $N_{3, 2}$. We begin by establishing the following proposition:

\begin{prop} \label{n32P}
Let $\OA^* = \left\{ \tau = (\tau_1, \tau_2, 0) \in \OA; \tau_2 > 0 \right\}$. Then the Hessian matrix of $\psi_*$ at $\tau \in \OA^*$, $\He \, \psi_*(\tau)$, is positive definite.
\end{prop}

The proof of Proposition \ref{n32P} is a slight modification of that of Proposition \ref{StarP}.

\medskip

By Proposition \ref{n32P} and Theorem \ref{TH2}, an explicit calculation yields the following:

\begin{theo} \label{n32T}
Let $x = |x| \, e_1$ with $|x| > 0$ and $\theta = (\theta_1, \theta_2, 0) \in \OA^*$. Set $t = \frac{|x|^2}{4} \nabla_{\theta} \psi_*(\theta)$, namely
\begin{align} \label{32nm}
t = \frac{|x|^2}{4} \, \theta_2 \left[ \frac{\psi'(|\theta|)}{|\theta|} \, \theta_2 \, \theta + 2 \, \psi(|\theta|) \, e_2 \right] \in \R \times (0, \, +\infty) \times \{ 0 \}.
\end{align}
Then we have
\begin{align} \label{dEn32}
d(x, t)^2 &= \phi((x, t); \theta) = |x|^2 \left[ 1 + \psi(|\theta|) \, \theta_2^2 + \psi'(|\theta|) \, |\theta| \, \theta_2^2 \right]  \nonumber \\
&= |x|^2 \left[ \frac{\theta_1^2}{|\theta|^2} + \left( \frac{\theta_2}{\sin{|\theta|}} \right)^2 \right] = \left| \frac{U(\theta)}{\sin{U(\theta)}} x \right|^2.
\end{align}
\end{theo}

\medskip

Unfortunately, Theorem \ref{n32T} determines the squared distance on ``half space'' not on the whole space. More precisely, let
\begin{align}
\OB := \{ (v_1, v_2) \in \R^2; \, v_2 > 0, \ v_1^2 + v_2^2 < \pi^2 \}
\end{align}
and
\begin{align}  \label{N32O}
\RN := \left\{ u = (u_1, u_2) \in \R^2; u_2 > \frac{2}{\sqrt{\pi}} \sqrt{|u_1|} \geq 0 \right\}.
\end{align}
See the sketch map in Figure \ref{figure2}.

\begin{figure}
\centering
\begin{overpic}[width = 15cm, height=7.5cm]{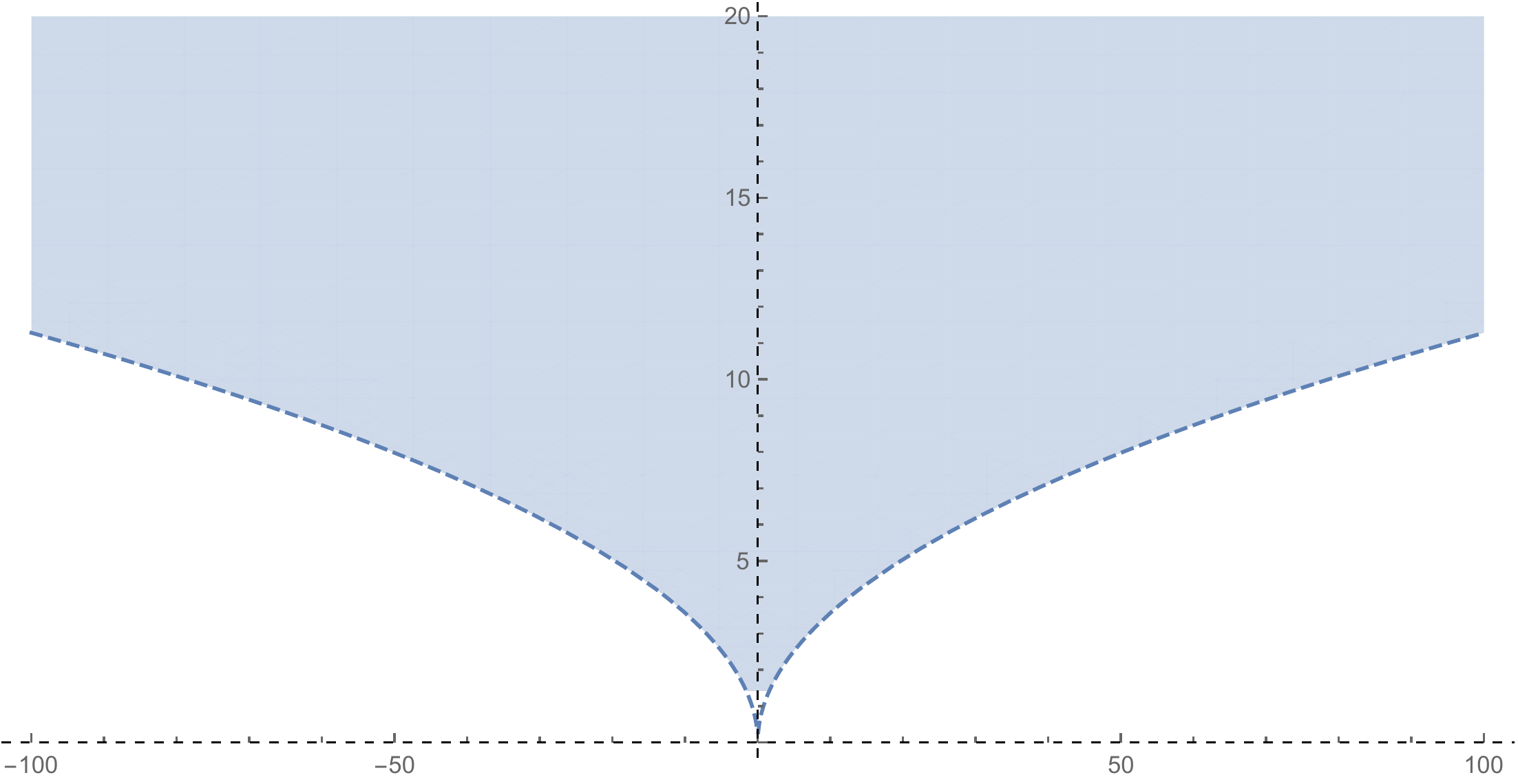}
\put(60,30){$\mathbb{R}_>^2$}
\put(101,2.5){$u_1$}
\put(49,52){$u_2$}
\put(15,12){$u_2 = \frac{2}{\sqrt{\pi}} \sqrt{|u_1|}$}
\put(28,16){$\nearrow$}
\end{overpic}
\caption[image]{Sketch map of $\mathbb{R}_>^2$  }\label{figure2}
\end{figure}

Consider now the smooth function $\MD$,
\begin{align} \label{32N96}
\MD(v_1, v_2) := v_2 \left[ \frac{\psi'(r)}{r} \, v_2 \, v + 2 \, \psi(r) \, e_2 \right], \qquad v = (v_1, v_2), \  r = |v|,
\end{align}
which is the variant of \eqref{32nm}. The following result is in fact a reformulation of Proposition \ref{Kp2}:

\begin{prop} \label{N32T2}
$\MD$ is a $C^{\infty}$-diffeomorphism from $\OB$ onto $\RN$.
\end{prop}

Naturally we ask what the exact expression of $d(|x| \, e_1, \frac{|x|^2}{4} (u, 0))^2$ is when $u \in \partial \RN$. We have the following answer:

\begin{cor} \label{Cn32}
Let $\alpha \ge 0$, $x = |x| e_1$ and $t = \frac{|x|^2}{4} (\pm\frac{\alpha^2}{\pi}, \frac{2}{\pi} \alpha, 0)$. Then we have
\begin{align} \label{N32dE}
d(x, t)^2 = |x|^2 + \alpha^2 |x|^2 = \phi((x, t); \pi \, e_1) = \sup_{\tau \in \OA} \phi((x, t); \tau).
\end{align}
Moreover, we have $d(0, t)^2 = 4 \pi |t|$ for any $t \in \R^3$.
\end{cor}

\begin{proof}
If $|x| = 0$ or $\alpha = 0$, \eqref{N32dE} is trivial. Assume now that $|x| \cdot \alpha > 0$. Furthermore, it follows from Lemma \ref{Ln91} that we can take $t = \frac{|x|^2}{4} (\frac{\alpha^2}{\pi}, \frac{2}{\pi} \alpha, 0)$. Observe, from the case (4) in the proof of Proposition \ref{Kp1}, that there exists $\{ v^{(j)} = (v^{(j)}_1, v^{(j)}_2) \} \subset \OB$ such that
\begin{align*}
u^{(j)} = \MD(v^{(j)}) \longrightarrow u^{(0)} = (\frac{\alpha^2}{\pi}, \frac{2}{\pi} \alpha)
\end{align*}
with $|v^{(j)}| \longrightarrow \pi^-$ and
\begin{align*}
\lim_{j \longrightarrow +\infty} 2 \, \psi(|v^{(j)}|) \, v^{(j)}_2 = \frac{2}{\pi} \alpha, \mbox{ \ i.e. \ } \lim_{j \longrightarrow +\infty} \frac{v^{(j)}_2}{\pi - |v^{(j)}|} = \alpha.
\end{align*}
Combining this with the first equality in \eqref{dEn32}, we obtain the desired result.

We are in a position to show that $d(0, t)^2 = 4 \pi |t|$ for $|t| > 0$. It suffices to consider the special case where $t = |t| e_2$. By the continuity of $d$, we have
\begin{align*}
d(0, |t| \, e_2)^2 = \lim_{\varepsilon \longrightarrow 0^+} d(\varepsilon \, e_1, |t| \, e_2)^2.
\end{align*}
Let $g_{\varepsilon} = (\varepsilon \, e_1, |t| \, e_2)$. By \eqref{32nm}, there exists $\theta_{\varepsilon} = (0, \theta_{\varepsilon, 2}, 0) \in \OA^*$ such that
\begin{align*}
\frac{4 |t|}{\varepsilon^2} = \left( |\theta_{\varepsilon}|^2 \psi(|\theta_{\varepsilon}|) \right)' = \frac{2 |\theta_{\varepsilon}| - \sin{(2 |\theta_{\varepsilon}|)}}{2 \sin^2{|\theta_{\varepsilon}|}}.
\end{align*}
Then it follows from the first equality in \eqref{dEn32} that
$d(g_{\varepsilon})^2 = \varepsilon^2 \left( \frac{|\theta_{\varepsilon}|}{\sin{|\theta_{\varepsilon}|}} \right)^2$,
which is exactly the squared distance in the setting of Heisenberg group, say $d_{\mathbb{H}(2, 1)}((\varepsilon, 0), |t|)^2$, cf. for example \cite{G77} or \cite{BGG00}. And it is well-known that
\begin{align*}
\lim_{\varepsilon \longrightarrow 0^+} d_{\mathbb{H}(2, 1)}((\varepsilon, 0), |t|)^2 = 4 \pi |t|.
\end{align*}

This completes the proof.
\end{proof}

\begin{remark}
Notice that $(0, t) \in \partial{\M}$ and $(e_1, \frac{1}{4} (u, 0)) \in \partial{\M}$ for $u \in \partial \RN$. Using the scaling property, the result of Corollary \ref{Cn32} can be deduced directly from Theorem \ref{THN6}.
\end{remark}

Set in the sequel the simply connected domain
\begin{align}
\RS := \left\{ (u_1, u_2) \in \R^2; \ u_1 > 0, \,  0 < u_2 < \frac{2}{\sqrt{\pi}} \sqrt{u_1} \right\}.
\end{align}
See the sketch map in Figure \ref{figure3}.

\begin{figure}
 \centering
\begin{overpic}[width = 12cm, height=7.5cm]{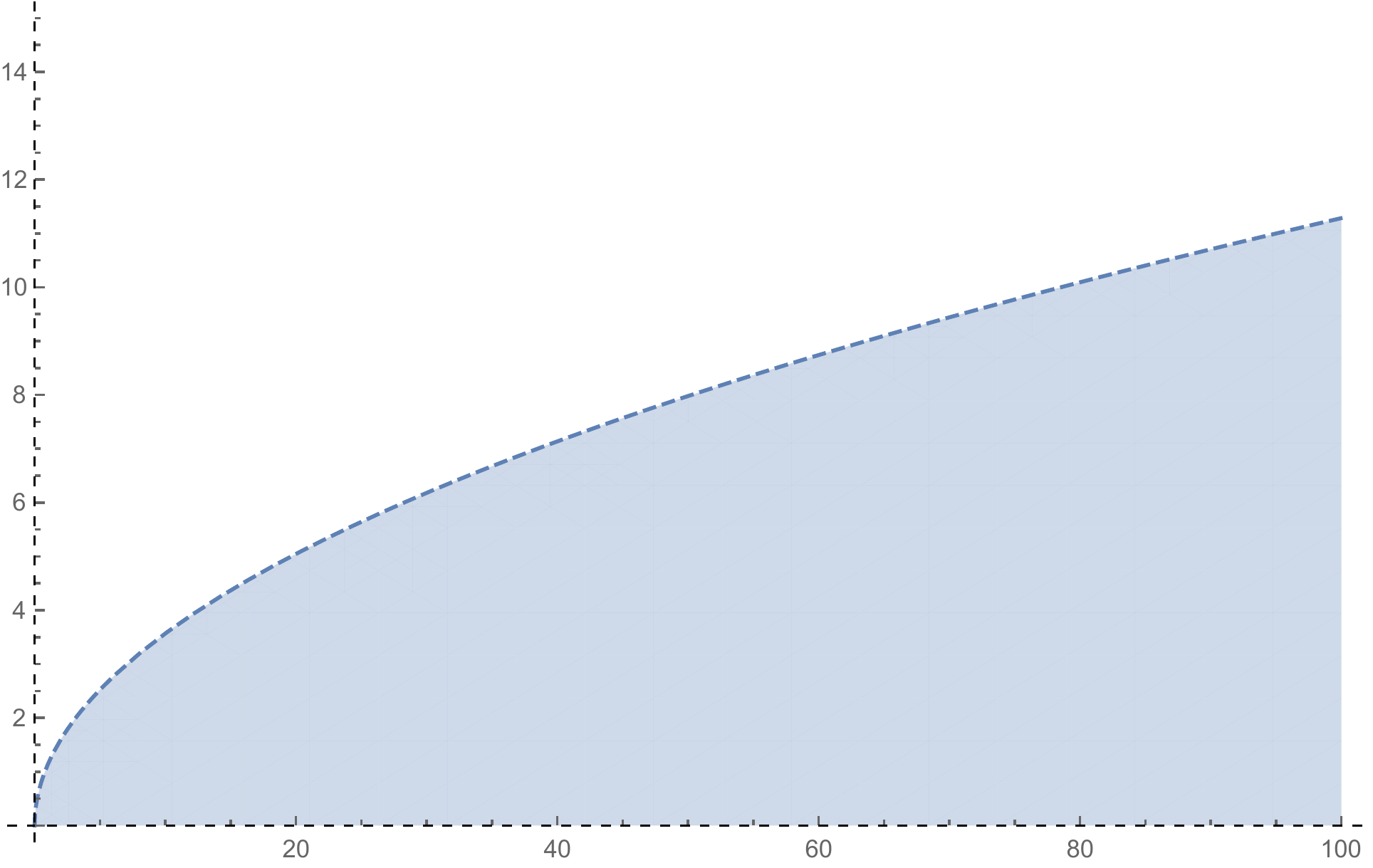}
\put(50,18){$\mathbb{R}_{<,+}^2$}
\put(1.5,64){$u_2$}
\put(101,2.5){$u_1$}
\put(15,33){$u_2 = \frac{2}{\sqrt{\pi}} \sqrt{u_1}$}
\put(29,29){$\searrow$}
\end{overpic}
\caption[image]{Sketch map of $\mathbb{R}_{<,+}^2$ }\label{figure3}
\end{figure}

By Lemma \ref{Ln91} and the scaling property (cf. \eqref{scaling}), it remains to determine
\begin{align}
d^2(|x| \, e_1, \frac{|x|^2}{4} (u, 0)) = |x|^2 \, d^2(e_1, \frac{1}{4} (u, 0)), \qquad |x| > 0, \ u \in \RS.
\end{align}

From now on, we will study the exact formulas of $d^2(e_1, \frac{1}{4} (u, 0))$ with $u \in \RS$. Inspired by the proof of Corollary \ref{Cn32}, we first seek suitable open sets $\mathcal{O}_1 \subset \RS$ and $\mathcal{O}_2 \subset \overline{B_{\R^2}(0, \pi)}^c$ which satisfy the following conditions:  \\
(1) $(\pi, 0) \in \partial \mathcal{O}_2$ and $\partial \mathcal{O}_1 \cap (\partial \RN \setminus \{ 0 \}) \neq \emptyset$;   \\
(2) $\Lambda$, defined by \eqref{32N96}, is a  $C^{\infty}$-diffeomorphism from $\mathcal{O}_2$ onto $\mathcal{O}_1$. Moreover, $\Lambda(v) \longrightarrow \partial \mathcal{O}_1 \cap (\partial \RN \setminus \{ 0 \})$ if $v \longrightarrow (\pi, 0)$ in $\mathcal{O}_2$ (this condition should be relaxed a little bit because of the next one); \\
(3) $\mathcal{O}_1$ is as much as possible in $\RS$.

(Obviously, we may have to repeat these processes in the general case of $\overline{\M} \subsetneqq \G$.)

To begin with, recall that $\V$ denotes the unique solution of $\tan{s} = s$ in $(\pi, \, \frac{3}{2} \pi)$, cf. \eqref{cN32}. Define
\begin{align} \label{nOnA}
\OC := \left\{ (v_1, v_2) \in \R^2; \, v_2 < 0, \, \pi < v_1 < r = \sqrt{v_1^2 + v_2^2} < \V, \  \K_3 < 0 \right\},
\end{align}
where
\begin{align} \label{32N98}
\K_3 := \K_3(v_1, v_2) := 2 \psi(r) + \frac{\psi'(r)}{r} v_2^2,
\end{align}

A sketch map of $\OC$ is exhibited in Figure \ref{figure4}.

\begin{figure}
 \centering
\begin{overpic}[width = 12cm, height=7.5cm]{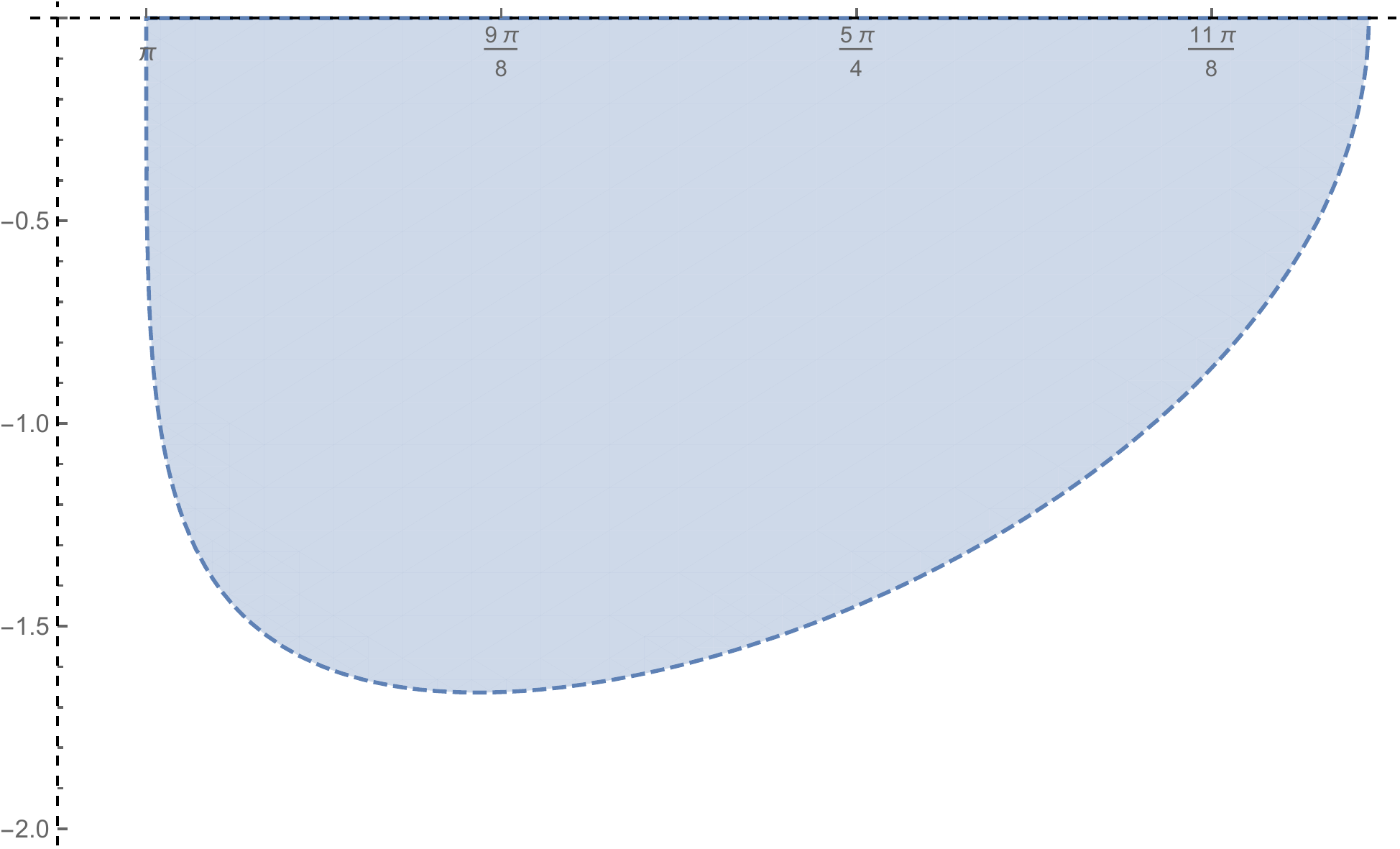}
\put(48,37){$\Omega_{-,4}$}
\put(2.5,-2.5){$v_2$}
\put(101,60.5){$v_1$}
\put(98.5,58){$\nwarrow$}
\put(100.5,55){$(\vartheta_1,0)$}
\put(68,18){$\nwarrow$}
\put(66,14){$\mathrm{K}_3 = 2\psi(r) + \frac{\psi^{\prime}(r)}{r} v_2^2 = 0$}
\end{overpic}
\caption[image]{Sketch map of $\Omega_{-,4}$ }\label{figure4}
\end{figure}
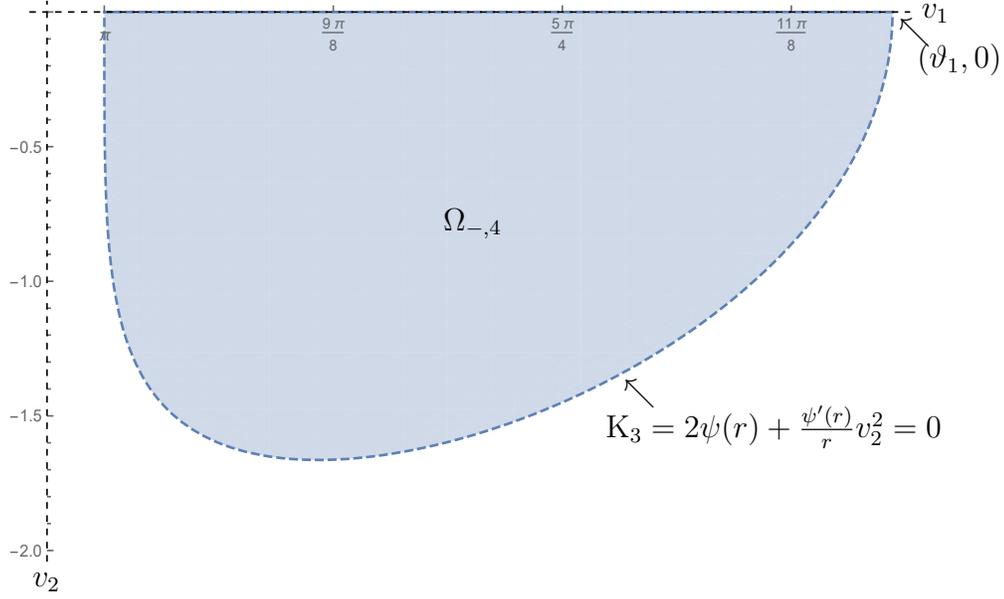

We have the following result.

\begin{theo} \label{N32T3}
$\MD$, defined by \eqref{32N96}, is a $C^{\infty}$-diffeomorphism from $\OC$ onto $\RS$. Furthermore, let $|x| > 0$ and $u = (u_1, u_2) \in \RS$. Set $x = |x| \, e_1$, $t = \frac{|x|^2}{4} (u, 0)$ and $\theta = (\MD^{-1}(u), 0)$. Then the Hessian matrix of $\phi((x, t); \cdot)$ at $\theta$ is nonsingular, and has exactly two positive eigenvalues.
\end{theo}

\begin{remark}
Combining Theorem \ref{N32T3} with Theorem \ref{nT122} below, at first glance, the last claim, namely $\He_{\theta} \phi((x, t); \theta)$ is indefinite, looks very strange in comparison with \eqref{csd}. Indeed, the highly non-trivial phenomenon can be explained by some ``inf-sup'' problem. Also, small-time asymptotics of the heat kernel on $N_{3, 2}$ can be completely solved. However, the asymptotic behavior at infinity of $p_1$ expressed in a simple form is currently unknown, we hope to solve it in a future work.
\end{remark}

Furthermore, we have the following:

\begin{theo} \label{nT122}
For any $u \in \RS$, we have
\begin{align} \label{FFF}
d^2(e_1, \frac{1}{4}(u, 0)) = \phi((e_1, \frac{1}{4}(u, 0)); (\Lambda^{-1}(u), 0)).
\end{align}
\end{theo}

Combining this with Lemma \ref{Ln91}, Theorem \ref{n32T}, Corollary \ref{Cn32}, the scaling property (see \eqref{scaling}) and \cite[Theorem~1.4]{MM17}, the Gaveau-Brockett optimal control problem on $N_{3, 2}$ is completely solved. It is worthwhile to point out that the main results obtained in \cite{MM17} can be deduced from Theorem \ref{nT122}, see \cite{LZ20} for more details.

\subsection{Proof of Theorem \ref{N32T3}} \label{nssn111}

First, the main difficulty for showing that $\MD$ is from $\OC$ onto $\RS$ is to check that
\begin{align*}
2 \psi(r) +  \frac{\psi'(r)}{r} v_2^2 > - \frac{2}{\sqrt{\pi}} \sqrt{\frac{\psi'(r)}{r} v_1}, \quad \forall (v_1, v_2) \in \OC.
\end{align*}
Using \eqref{32ei3}, it remains to prove that $\psi(r) > - \sqrt{\frac{\psi'(r)}{r}}$ which is obvious by \eqref{32ei2} and \eqref{32ei3}.

Moreover, using an argument similar to that in the proof of Proposition \ref{Kp1}, it is easy to show that the smooth function $\MD$ is proper. More precisely, for $j \in \N^*$, let $v^{(j)} = (v_1^{(j)}, v_2^{(j)}) \in \OC$ and $u^{(j)} = (u_1^{(j)}, u_2^{(j)}) := \Lambda(v^{(j)})$, we have that:
\begin{enumerate}[(I)]
  \item If $v^{(j)} \not\longrightarrow (\pi, \, 0)$ with $\K_3(v^{(j)}) \longrightarrow 0^-$ or $v_2^{(j)} \longrightarrow 0^-$, then $u^{(j)} \longrightarrow \partial \RS$ since $u_2^{(j)} \longrightarrow 0^+$;

  \item If $v^{(j)} \longrightarrow (\pi, \, 0)$ with $u_1^{(j)} \longrightarrow 0^+$ or $+\infty$, it is trivial;

  \item If $v^{(j)} \longrightarrow (\pi, \, 0)$ with $u_1^{(j)} \longrightarrow a > 0$, it follows from Lemma \ref{Ll34} that
\begin{align*}
\lim_{r \longrightarrow \pi^+} (\pi - r) \psi(r) = \frac{1}{\pi}, \qquad \lim_{r \longrightarrow \pi^+} (\pi - r)^2 \psi'(r) = \frac{1}{\pi}.
\end{align*}
A simple calculation shows that $u_2^{(j)} \longrightarrow \frac{2}{\sqrt{\pi}} \sqrt{a}$, so $u^{(j)} \longrightarrow \partial \RS$.
\end{enumerate}

By means of Hadamard's theorem, it suffices to prove that the Jacobian determinant of $\MD$ vanishes nowhere. Indeed, the Jacobian matrix of $\MD$ at $v = (v_1, v_2)$ equals
\begin{align} \label{nJAn}
\left(
\begin{array}{cc}
v_2^2 \left[ \frac{\psi'(r)}{r} + \left( \frac{\psi'(r)}{r} \right)' \frac{v_1^2}{r} \right] \qquad & \qquad v_1 v_2 \left[ 2 \frac{\psi'(r)}{r} + \left( \frac{\psi'(r)}{r} \right)' \frac{v_2^2}{r}  \right] \\
\mbox{} \qquad & \qquad \mbox{} \\
v_1 v_2 \left[ 2 \frac{\psi'(r)}{r} + \left( \frac{\psi'(r)}{r} \right)' \frac{v_2^2}{r}  \right] \qquad & \qquad  2 \psi(r) + v_2^2 \left[ 5 \frac{\psi'(r)}{r} + \left( \frac{\psi'(r)}{r} \right)' \frac{v_2^2}{r} \right]  \\
\end{array}
\right),
\end{align}
saying $\JA(v, r)$. Recall that (see \eqref{32ei3}, \eqref{32ei4} and \eqref{32N98})
\begin{align*}
\K_1 := \frac{\psi'(r)}{r}, \quad \K_2 := r^{-1} \left( \frac{\psi'(r)}{r} \right)', \quad \K_3 := 2 \psi(r) + \K_1 v_2^2,
\end{align*}
thus $\det \JA(v, r)$ equals
\[
v_2^2 \left\{ (\K_1 + \K_2 v_1^2) [\K_3 + v_2^2 (4 \K_1 + \K_2 v_2^2)] - v_1^2 (2 \K_1 + \K_2 v_2^2)^2 \right\} := v_2^2 \, \K.
\]

Observe that
\begin{align} \label{32nK}
\K &= (2 \psi(r) + \K_1 v_2^2) \cdot (\K_1 + v_1^2 \K_2) + 4 \K_1^2 v_2^2 + \K_1 \K_2 v_2^4 - 4 \K_1^2 v_1^2  \nonumber \\
&= 2  \psi(r) \K_1 + v_1^2 \left[ 2 \psi(r) \K_2 - 4 \K_1^2 \right] + v_2^2 \, \K_1 \left( 5 \K_1 + \K_2 \, r^2 \right).
\end{align}

To finish the proof of the first claim in  Theorem \ref{N32T3}, it only needs to prove the following

\begin{lem} \label{nL112}
We have $\K < 0$ for any $v = (v_1, v_2) \in \OC$.
\end{lem}

\begin{proof}
We start with the case $5 \K_1 + \K_2 \, r^2 \leq 0$. It follows from \eqref{32ei2}, \eqref{32ei3} and \eqref{32ei5} that $\K < 0$.

To treat the opposite case $5 \K_1 + \K_2 \, r^2 > 0$, the condition $\K_3 = 2 \psi(r) + v_2^2 \, \K_1 < 0$ implies that
\begin{align*}
\K  &<  2  \psi(r) \K_1 + v_1^2 \left[ 2 \psi(r) \K_2 - 4 \K_1^2 \right] - 2 \psi(r) \left( 5 \K_1 + \K_2 \, r^2 \right) \\
&= - 2 \left\{ 2 \K_1^2 v_1^2 +  \psi(r)  \left[ 4 \K_1 + \K_2 v_2^2 \right] \right\}.
\end{align*}
It remains to show that $2 \K_1^2 v_1^2 >  -  \psi(r)  \left( 4 \K_1 + \K_2 v_2^2 \right)$. By \eqref{32ei2} and \eqref{32ei3}, via the trivial inequality $(a + b)^2 \geq a (a + 2 b)$, it suffices to establish the estimate
\begin{align*}
&  \frac{v_1^2}{(r^2 - \pi^2)^2} \left\{ (r^2 - \pi^2)^{-2} + 2 \sum_{j = 2}^{+\infty}  \left[(j \, \pi)^2 - r^2 \right]^{-2} \right\} \\
&> \frac{1}{r^2 - \pi^2} \left\{ \frac{1}{(r^2 - \pi^2)^2} +  \sum_{j = 2}^{+\infty}  \frac{1}{\left[(j \, \pi)^2 - r^2 \right]^2} - \frac{v_2^2}{ (r^2 - \pi^2)^3} + \sum_{j = 2}^{+\infty}  \frac{v_2^2}{\left[(j \, \pi)^2 - r^2 \right]^3} \right\} \\
&= \frac{1}{r^2 - \pi^2} \left\{ \frac{v_1^2 - \pi^2}{(r^2 - \pi^2)^3} + \sum_{j = 2}^{+\infty}  \frac{(j \, \pi)^2 - v_1^2}{\left[(j \, \pi)^2 - r^2 \right]^3} \right\}.
\end{align*}

To obtain the last bound, observe first that for $j \ge 3$
\begin{align*}
\frac{2 v_1^2}{r^2 - \pi^2} > \frac{2 \pi^2}{(\frac{3}{2} \pi)^2 - \pi^2} > \frac{(3 \pi)^2 - \pi^2}{(3 \pi)^2 - (\frac{3}{2} \pi)^2} >  \frac{(j \pi)^2 - v_1^2}{(j \pi)^2 - r^2}.
\end{align*}
Moreover, by the fact that $r^2 - \pi^2 < (2 \pi)^2 - r^2$, we have
\begin{align*}
\frac{\pi^2}{(r^2 - \pi^2)^3} + \frac{2 v_1^2}{r^2 - \pi^2} \frac{1}{[(2 \pi)^2 - r^2]^2} > \frac{\pi^2 + 2 v_1^2}{\left[(2 \pi)^2 - r^2 \right]^3} > \frac{(2 \pi)^2 - v_1^2}{\left[(2 \pi)^2 - r^2 \right]^3},
\end{align*}
where we have used  $v_1 > \pi$ in the last inequality. Then
\begin{align*}
\frac{v_1^2}{(r^2 - \pi^2)^2} \left\{ \frac{1}{(r^2 - \pi^2)^2} + \frac{2}{[(2 \pi)^2 - r^2 ]^2} \right\} >  \frac{1}{r^2 - \pi^2} \left\{ \frac{v_1^2 - \pi^2}{(r^2 - \pi^2)^3} + \frac{(2 \pi)^2 - v_1^2}{\left[(2 \pi)^2 - r^2 \right]^3} \right\}.
\end{align*}

This completes the proof.
\end{proof}

The second claim in Theorem \ref{N32T3} is now clear. Indeed, setting $\Lambda^{-1}(u) = (\theta_1, \theta_2) := \widetilde{\theta}$, a direct calculation shows that $\He_{\tau = \theta} \phi((x, t); \tau)$ equals
\begin{align*}
- |x|^2 \, \left(
          \begin{array}{cc}
            \JA(\widetilde{\theta}, |\theta|)  & \mathbb{O}_{1 \times 2} \\
             \mathbb{O}_{2 \times 1} & 2 \psi(|\theta|) + \theta_2^2 \frac{\psi'(|\theta|)}{|\theta|} = \K_3(\widetilde{\theta}) \\
          \end{array}
        \right),  \quad \mbox{with $\JA$ defined by \eqref{nJAn}, }
\end{align*}
which is nonsingular, and has exactly two positive eigenvalues because of $\K_3(\widetilde{\theta}) < 0$ (cf. \eqref{nOnA}-\eqref{32N98}) and Lemma \ref{nL112}. This ends the proof of Theorem \ref{N32T3}.  \qed

\medskip

\subsection{Proof of Theorem \ref{nT122}} \label{ss123}

\medskip

\begin{proof}
There are three steps.

\textit{Step 1.} It follows from \eqref{n32c} that
\[
d(e_1, 0)^2 = 1 = \phi((e_1, 0); e_1) = \phi((e_1, 0); - e_1),
\]
which implies $(e_1, 0) \in \OM_2$. Up to an orthogonal transform and scaling, we have
\begin{align} \label{n32abn}
\left\{ (x, 0); \ x \in \R^3 \right\} \subseteq \OM_2 \subseteq \mathrm{Abn}_o^* \subseteq \mathrm{Abn}_o = \left\{ (x, 0); \ x \in \R^3 \right\},
\end{align}
where we have used Corollary \ref{nCCn} in the second inclusion, the definition of $\mathrm{Abn}_o^*$ in the third inclusion and \cite{LDMOPV16} in the equality. In other words, the abnormal shortest geodesics are straight segments in the first layer.

Next, by \eqref{n32abn} and \cite[Theorem 1]{LZ20}, it follows from \cite{MM17} that
\begin{align} \label{N32Cut}
\mathrm{Cut}_{o} = \{(x, t); \ \mbox{$x$ and $t$ are linearly dependent}\}.
\end{align}
Notice that our definition of $\mathrm{Cut}_o$ is clearly different from that in \cite{MM17}.

\textit{Step 2.} Let $\alpha > 0$ and $g_{\alpha} := (e_1, \frac{1}{4} (\frac{\alpha^2}{\pi}, \frac{2}{\pi} \alpha, 0))$. Let us consider a $C^{\infty}$ function on $N_{3, 2}$, $\phi_{\alpha}$, such that
\[
\phi_{\alpha}(x, t) = |x|^2 + 4 \pi t \cdot \frac{x}{|x|} \quad \mbox{for $(x, t)$ near $g_{\alpha}$.}
\]
By \eqref{n32c}, we have $\phi_{\alpha}(x, t) = \phi((x, t); \pi \frac{x}{|x|})$. Then for $g = (x, t)$ near $g_{\alpha}$, Theorem \ref{MP1} and Remark \ref{nRk21} imply that $\phi_{\alpha}(g) \le d(g)^2$. Moreover it follows from \eqref{N32dE} that $\phi_{\alpha}(g_{\alpha}) = d(g_{\alpha})^2$. Hence, we have (cf. \eqref{PSD}) $d\phi_{\alpha}(g_{\alpha}) = 2 \, z_{\alpha}  \in \partial_P d^2(g_{\alpha})$ where
\begin{align}
z_{\alpha} := (\zeta_{\alpha}, \ 2 \, \pi \, e_1) \in \R^3 \times \R^3, \qquad \zeta_{\alpha} := (1, \alpha, 0).
\end{align}

From \cite[Proposition~2]{RT05}, there exists a unique shortest geodesic from $o$ to $g_{\alpha}$, $\gamma_{g_{\alpha}}$, which admits a normal extremal lift $(\gamma_{g_{\alpha}}(s), (\xi(s), \tau(s)))$ such that $(\xi(1), \tau(1)) = z_{\alpha}$. By \eqref{HaEP}, we have $\zeta(1) = \xi(1) = \zeta_{\alpha}$. Moreover, a simple calculation shows that
\begin{align} \label{nMATRIXn}
e^{-\widetilde{U}(\pi \, e_1)} = \left(
  \begin{array}{ccc}
    1 & 0 & 0 \\
    0 & -1 & 0 \\
    0 & 0 & -1 \\
  \end{array}
\right), \quad \mbox{so } e^{-2 \widetilde{U}(\pi \, e_1)} = \mathbb{I}_3.
\end{align}
Then, by \eqref{GEn3*}, $\zeta(0) = \zeta(1) = \zeta_{\alpha}$. In conclusion, via
\eqref{nFG}, we have that
\[
g_{\alpha} = \exp\{(\zeta_{\alpha}, 2 \, \pi \, e_1) \}, \quad \gamma_{g_{\alpha}}(s) = \exp\{s \, (\zeta_{\alpha}, 2 \, \pi \, e_1) \}.
\]

It follows from \eqref{N32Cut} that $g_{\alpha} \not\in \mathrm{Cut}_o$. By Lemma \ref{AxL}, there exists a neighborhood $V_{(\zeta_{\alpha}, 2 \, \pi \, e_1)}$ of $(\zeta_{\alpha}, 2 \, \pi \, e_1)$ such that the sub-Riemannian exponential map is a diffeomorphism from $V_{(\zeta_{\alpha}, 2 \, \pi \, e_1)}$ onto its image $O_{g_{\alpha}} := \exp{(V_{(\zeta_{\alpha}, 2 \, \pi \, e_1)})} \subset (\mathrm{Cut}_o)^c$. In addition, we have
\begin{align} \label{n32dd}
d^2\left( \exp\{ (\zeta, 2 \, \theta) \} \right)  = |\zeta|^2, \qquad \forall (\zeta, 2 \, \theta) \in V_{(\zeta_{\alpha}, 2 \, \pi \, e_1)}.
\end{align}

Let $u \in \RS$. Set, via the first claim in Theorem \ref{N32T3},
\[
\theta_u := (\Lambda^{-1}(u), 0) := (\theta_1(u), \theta_2(u), 0).
\]
Recall that $\mathcal{V}$ is defined by \eqref{nD1V}. It is clear that $\theta_u \in \mathcal{V}$ since the eigenvalues of $U(\theta_u)$ are $\{0, \pm |\Lambda^{-1}(u)|\}$ with $\pi < |\Lambda^{-1}(u)| < \vartheta_1 < \frac{3}{2} \pi$. By Theorem \ref{THN4} and \eqref{iniC}, we take
\[
\widetilde{\zeta}_u := e^{- \widetilde{U}(\theta_u)} \frac{U(\theta_u)}{\sin{U(\theta_u)}} \, e_1 \quad \mbox{ such that } \quad \exp\{(\widetilde{\zeta}_u, 2 \, \theta_u)\} = (e_1, \frac{1}{4} (u, 0)) := \widetilde{g}_u.
\]

Assume for a moment that
\begin{align} \label{N32cl}
(\widetilde{\zeta}_u, 2 \, \theta_u) \in V_{(\zeta_{\alpha}, 2 \, \pi \, e_1)} \quad \mbox{for $u \in \RS$ near $(\frac{\alpha^2}{\pi}, \frac{2}{\pi} \alpha)$.}
\end{align}
Hence it follows from \eqref{n32dd} and the equality in \eqref{UED}
that
\begin{align*}
d^2(\widetilde{g}_u) = |\widetilde{\zeta}_u|^2 = \phi(\widetilde{g}_u; \theta_u), \quad \mbox{for some $u \in \RS$ near $(\frac{\alpha^2}{\pi}, \frac{2}{\pi} \alpha)$.}
\end{align*}

We now prove \eqref{N32cl}. Indeed, from the proof, in Subsection \ref{nssn111}, of the fact that $\Lambda$ is proper, we have that
\[
\lim_{u \in \RS, u \longrightarrow (\frac{\alpha^2}{\pi}, \frac{2}{\pi} \alpha)} \theta_u = \pi \, e_1 \quad \mbox{ and } \quad \lim_{u \in \RS, u \longrightarrow (\frac{\alpha^2}{\pi}, \frac{2}{\pi} \alpha)} \frac{\theta_2(u)}{\pi - |\theta_u|} = \alpha.
\]
Then, using \eqref{n32na}, a simple calculation yields that
\[
\lim_{u \in \RS, u \longrightarrow (\frac{\alpha^2}{\pi}, \frac{2}{\pi} \alpha)} \frac{U(\theta_u)}{\sin{U(\theta_u)}} \, e_1 = (1, - \alpha, 0)^{\mathrm{T}}.
\]
Combining this with \eqref{nMATRIXn}, we obtain
\[
\lim_{u \in \RS, u \longrightarrow (\frac{\alpha^2}{\pi}, \frac{2}{\pi} \alpha)} \widetilde{\zeta}_u = \zeta_{\alpha}.
\]

\textit{Step 3.} By the second claim in Theorem \ref{N32T3}, \eqref{N32Cut} and Corollary \ref{nC124}, the nonempty set $\mathbf{S} := \left\{u \in \RS; \ d^2(\widetilde{g}_u) = \phi(\widetilde{g}_u; \theta_u) \right\}$ is open. Moreover Theorem \ref{N32T3} shows that $\phi(\widetilde{g}_u; \theta_u)$ is smooth in $\RS$. Recall that the squared sub-Riemannian distance is continuous. Then $\mathbf{S}$ is also closed in $\RS$. By the connectness of $\RS$, we get $\mathbf{S} = \RS$, which finishes the proof of Theorem \ref{nT122}.
\end{proof}

\begin{remark}
(1) It is not hard to show directly $\OM_2 = \left\{ (x, 0); \ x \in \R^3 \right\}$. Then it follows from \cite{LZ20} that $\mathrm{Abn}_o^* = \left\{ (x, 0); \ x \in \R^3 \right\}$.

(2) By using the smooth function $\phi_{\alpha}$, we can show directly $g_{\alpha} \not\in \mathrm{Cut}_o$. See \cite[the beginning of Case (3) in the proof of Proposition 9]{LZ20} for more details.

(3) A completely different, more self-contained proof via Corollary \ref{nCx} can be found in \cite{LZ20}.
\end{remark}

\section*{Acknowledgement}
\setcounter{equation}{0}
This work is partially supported by NSF of China (Grants  No. 11625102 and No. 11571077).
I am indebted to Alano Ancona for the proof  of Proposition \ref{NPA}.

\bigskip

\mbox{}\\
Hong-Quan Li\\
School of Mathematical Sciences/Shanghai Center for Mathematical Sciences  \\
Fudan University \\
220 Handan Road  \\
Shanghai 200433  \\
People's Republic of China \\
E-Mail: hongquan\_li@fudan.edu.cn

\end{document}